\documentclass[11pt,reqno]{amsart}

\usepackage{amssymb}
\usepackage{latexsym}
\usepackage{amsmath}

\def\wt{\widetilde}
\def\ov{\overline}
 \def\up{\upharpoonright}
\def\cH{\mathcal H}

\def\cD{\mathcal D}  \def\cR{\mathcal R} 
\def\cP {\mathcal P}
\def\cK{\mathcal K} \def\cL{\mathcal L}  \def\cZ{\mathcal Z}
\def\cJ{\mathcal J}

\def \gH{\mathfrak H}   \def \gN{\mathfrak N}

\def \bC{\mathbb C}  
\def\bR{\mathbb R}
\def \l{\lambda}
\def \a{\alpha} \def \b{\beta}    \def \s{\sigma}
\def \t{\theta} \def\g {\gamma}
\def\d {\delta} \def\e{\varepsilon} \def\om {\omega} \def\Om {\Omega}
\def \f{\varphi}  \def \G{\Gamma} \def\D {\Delta} \def\S{\Sigma}

\def \C{\widetilde {\mathcal C}}
\def \CA{\C(\cH_0,\cH_1)}

\def \CG{\C(\mathcal H)}

\def \cd {\cdot}

\def \Ker{\text{Ker}} \def\codim{\text{codim}\,}
 
  \def \exl { {Ext}_{L_0}}
 
\def\tm{\times}
\def  \RH {\wt R (\cH_0,\cH_1)} \def \RZ {\wt R^0 (\cH_0,\cH_1)}
   \def\HH {H^n\oplus H^n}
\def \pair {\tau=\{\tau_+,\tau_-\}}

\newcommand {\LH}[1] {L'_2[#1,H ]}
  \def \RH {\wt R (\cH_0,\cH_1)}
\def \RZ {\wt R^0 (\cH_0,\cH_1)}

\def\bt{\{\cH,\G_0,\G_1\}}
\def\bta{\{\cH_0\oplus \cH_1,\Gamma _0,\Gamma _1\}}

\def \R { R (\l)}  \def\CD {\{C(\cd), D(\cd)\}}
\def\TR {TR\{\cH_0,\cH_1\}} \def\TRZ {TR^0\{\cH_0,\cH_1\}}
\def \CR {\bC\setminus\bR} \def\ml {m_\cP(\cd)}

\newtheorem{theorem}{Theorem}[section]
\newtheorem{proposition}[theorem]{Proposition}
\newtheorem{corollary}[theorem]{Corollary}
\newtheorem{lemma}[theorem]{Lemma}

\theoremstyle{definition}
\newtheorem {definition} [theorem]{Definition}
\theoremstyle{remark}
\newtheorem{remark}[theorem]{Remark}
\numberwithin{equation}{section}

\numberwithin{equation}{section}
\begin{document}
\title [Minimal spectral functions]
{ Minimal spectral functions of an ordinary differential
operator}
\author {Vadim Mogilevskii}
\address{Department of Calculus\\
Lugans'k National   University  \\
2 Oboronna, Lugans'k, 91011\\
Ukraine} \email{vim@mail.dsip.net}


\begin{abstract}
Let $l[y]$ be a formally selfadjoint differential expression of an even order
on the interval $[0,b\rangle \;(b\leq \infty)$ and let $L_0$ be the
corresponding minimal operator. By using the concept of a decomposing boundary
triplet we consider the boundary problem formed by the equation $l[y]-\l
y=f\;(f\in L_2[0,b\rangle)$ and the Nevanlinna $\l$-depending boundary
conditions with constant values at the regular endpoint $0$. For such a problem
we introduce the concept of the $m$-function, which in the case of selfadjoint
decomposing boundary conditions coincides with the classical characteristic
(Titchmarsh-Weyl) function. Our method allows one to describe all minimal
spectral functions of the boundary problem, i.e., all spectral functions of the
minimally possible dimension. We also improve (in the case of intermediate
deficiency indices $n_\pm(L_0)$ and not decomposing boundary conditions) the
known estimate of the spectral multiplicity of the (exit space) selfadjoint
extension $\wt A\supset L_0$. The results of the paper are obtained for
expressions $l[y]$ with operator valued coefficients and arbitrary (equal or
unequal) deficiency indices $n_\pm(L_0)$.
\end{abstract}
\maketitle
\section{Introduction}
The main objects of the paper are differential operators generated by a
formally selfadjoint differential expression $l[y]$ of an even order $2n$ on an
interval $\D=[0,b\rangle  \;(b\leq \infty)$. We consider the expression $l[y]$
with operator valued coefficients and arbitrary (possibly unequal) deficiency
indices, but in order to simplify presentation of the main results assume that
\begin {equation}\label{1.1}
l[y]=\sum_{k=1}^n (-1)^k (p_{n-k}y^{(k)})^{(k)}+p_n y
\end{equation}
is a scalar expression with real-valued coefficients $p_k(t) \; (t\in \D)$
\cite{Nai}. Denote by $L_0$ and $L(=L_0^*)$ minimal and maximal operators
 respectively generated  by the expression \eqref{1.1} in the Hilbert
space $\gH:=L_2(\D)$ and let $\cD$ be the domain of $L$. As is known $L_0$ is a
symmetric operator with equal deficiency indices $m=n_\pm (L_0)$ and $n\leq m
\leq 2n$. Denote also by $n_b:=m-n$ the defect number of the expression
\eqref{1.1} at the point $b$ \cite{Mog09.1}.

In the present paper we develop an approach based on the concept of a
decomposing boundary triplet for a differential operator \cite{Mog09.1,Mog09.2,
Mog10}. Recall that according to \cite{Mog09.1} a decomposing boundary triplet
for $L$ is a boundary triplet $\Pi=\{\bC^n\oplus \bC^{n_b}, \G_0,\G_1\}$ in the
sense of \cite{GorGor} with the  boundary operators $\G_j:\cD\to \bC^n\oplus
\bC^{n_b},\; j\in\{0,1\}$  of the special form
\begin {equation}\label{1.2}
\G_0 y=\{y^{(2)}(0), \G'_0y\}\,(\in \bC^n\oplus \bC^{n_b}), \quad \G_1
y=\{-y^{(1)}(0), \G'_1y\}\,(\in \bC^n\oplus \bC^{n_b}).
\end{equation}
Here $y^{(j)}(0)$ are vectors of quasi-derivatives \eqref{3.2} at the point $0$
and $\G_j'y(\in \bC^{n_b}), \; j\in\{0,1\}$ are vectors of boundary values of a
function $y\in\cD$ at the singular endpoint $b$.

Next assume that $\cP=\{C_0(\l), C_1(\l)\}\;(\l\in\CR)$ is a Nevanlinna
operator pair defined by the block representations
\begin {equation}\label{1.3}
C_0(\l)=(\hat C_0\;\; C_0'(\l)):\bC^n\oplus \bC^{n_b}\to \bC^m, \;\;
C_1(\l)=(\hat C_1\;\; C_1'(\l)):\bC^n\oplus \bC^{n_b}\to \bC^m
\end{equation}
with the constant entries $\hat C_0, \; \hat C_1$ and  let $\tau=\tau(\l):=
\{\{h,h'\} :C_0(\l)h+ C_1(\l)h'=0\}$ be the corresponding Nevanlinna family of
linear relations. Denote by $\hat\cK$ the range of the operator $\hat C  =(\hat
C_0\;\; \hat C_1)$ and let
\begin {equation*}
\hat n=\dim \hat\cK=\text{rank} (\hat C_0\;\; \hat C_1), \qquad n'=m-\hat n.
\end{equation*}
Then $n\leq \hat n\leq m$ and the operator pair \eqref{1.3} admits the block
representation
\begin {gather}
C_0(\l)=\begin{pmatrix} N_0 & C'_{01}(\l)\cr 0 & C'_{02}(\l)
\end{pmatrix}:\bC^n\oplus \bC^{n_b}\to \hat\cK\oplus \hat\cK^\perp ,\label{1.4}\\
C_1(\l)=\begin{pmatrix} N_1 & C'_{11}(\l)\cr 0 &
C'_{12}(\l\end{pmatrix}):\bC^n\oplus \bC^{n_b}\to \hat\cK\oplus \hat\cK^\perp
,\label{1.4a}
\end{gather}
where $N_j$ are $(\hat n\times n )$-matrices with $\text{rk} (N_0\;\;N_1)=\hat
n $ and $C_{j1}'(\l), \; C_{j2}'(\l)\;(j\in \{0,1\})$ are respectively $(\hat
n\times n_b)$ and $(n'\times n_b)$-matrix functions. By using the boundary
operators \eqref{1.2} consider the boundary problem
\begin{gather}
l[y]-\l y=f \label{1.5}
\end{gather}
$$
 C_0(\l) \G_0 y-C_1(\l)\G_1 y =0.\leqno (*)
$$
It follows from \eqref{1.4}, \eqref{1.4a}  that the boundary condition ($*$)
can be written as two equalities
\begin{gather}
N_0 y^{(2)}(0)+N_1 y^{(1)}(0)+C_{01}'(\l)\G'_0 y- C_{11}'(\l)\G'_1
y=0,\label{1.7}\\
C_{02}'(\l)\G'_0 y- C_{12}'(\l)\G'_1 y=0,\label{1.8}
\end{gather}
which define in fact $m$ linearly independent  boundary conditions in the sense
of \cite{DunSch}.

The problem \eqref{1.5}-\eqref{1.8} is a particular case of a general
Nevanlinna type boundary problem  and hence it generates a generalized
resolvent $R(\l)=R_\tau(\l)$ and the corresponding spectral function
$F(t)=F_\tau (t) $ of the operator $L_0$ \cite{Mog10}. Moreover each
selfadjoint boundary problem is given by the boundary condition ($*$) with a
constant-valued Nevanlinna pair $\cP=\{C_0,C_1\}$, which implies that each
canonical resolvent of the operator $L_0$ is generated by the boundary problem
\eqref{1.5}-\eqref{1.8} with $C'_{j1}(\l)\equiv C'_{j1}$ and $C'_{j2}(\l)\equiv
C'_{j2}, \; j\in\{0,1\},\;\l\in\CR$. Observe also that the problem
\eqref{1.5}-\eqref{1.8} contains as a particular case a decomposing boundary
problem. Namely if (and only if) the equality $\hat n=n$ is satisfied, then
$C_{01}'(\l)= C_{11}'(\l)=0$ and the boundary conditions \eqref{1.7},
\eqref{1.8} becomes decomposing.

Next assume that $M(\cd)$ is the Weyl function of the decomposing boundary
triplet \eqref{1.2} in the sense of \cite{DM91} and let
\begin {equation}\label{1.9}
M(\l)=\begin{pmatrix}m(\l) & M_2(\l) \cr M_3(\l) & M_4(\l)
\end{pmatrix}: \bC^n\oplus \bC^{n_b}\to \bC^n\oplus \bC^{n_b}, \quad \l\in\CR
\end {equation}
be the block representation of $M(\l)$. Moreover let $\Om (\l)=\Om_\tau(\l)$ be
the Shtraus characteristic matrix of the generalized resolvent $R(\l)=R_\tau
(\l)$ \cite{Sht57}. Then $\Om_\tau(\l)$ is defined immediately in terms of a
Nevanlinna boundary parameter $\tau$ by the equalities
\begin {gather}
\wt\Om_{\tau }(\l)=\begin{pmatrix} M(\l)-M(\l)(\tau(\l) +M(\l))^{-1}M(\l) &
-\tfrac 1 2 I +M(\l) (\tau(\l) +M(\l))^{-1} \cr -\tfrac 1 2 I
+(\tau(\l) +M(\l))^{-1}M(\l) & -(\tau(\l) +M(\l))^{-1}\end{pmatrix},\label{1.9a}\\
\Om_\tau (\l)= P_{\bC^n\oplus \bC^n}\,\wt \Om_{\tau }(\l)\up \bC^n\oplus \bC^n,
\;\; \l\in\CR.\nonumber
\end{gather}
(see \cite{Mog10}).For a given operator pair \eqref{1.4}, \eqref{1.4a} consider
also the operator function
\begin{gather}\label{1.10}
\Om_{\tau,W'}(\l)=(W')^{-1}\Om_\tau(\l)(W')^{-1*}, \quad \l\in\CR,
\end{gather}
where $ W'=\begin{pmatrix} -N_0^* & * \cr N_1^* & * \end{pmatrix} $ is an
invertible $(2n\times 2n)$-matrix (the form of the entries $*$ does not
matter). We show in the paper that  the operator function \eqref{1.10} is of
the form
\begin {equation}\label{1.11}
\Om_{\tau,W'}(\l)=\begin{pmatrix} m_\cP(\l) & C^* \cr C & 0
\end{pmatrix}:\hat\cK\oplus \hat\cK^\perp \to \hat\cK\oplus \hat\cK^\perp,
\quad \l\in\CR,
\end {equation}
where $C$ is a  constant operator. The equality \eqref{1.11} generates the
uniformly strict Nevanlinna operator function $m_\cP(\cd)$,  which we call an
$m$-function of the boundary problem \eqref{1.5}-\eqref{1.8}. This function can
be also explicitly   defined in terms of the boundary conditions \eqref{1.7},
\eqref{1.8} (see Theorem \ref{th4.11}, 3)). Moreover in the case of selfadjoint
decomposing boundary conditions  the function $m_\cP(\cd)$ coincides with the
classical characteristic (Titchmarsh-Weyl) function \cite{Nai}.

It turns out that the characteristic matrix  $\Om_\tau(\cd)$ and the
$m$-function $m_\cP(\cd)$ are connected by
\begin {equation*}
m_\cP(\l)=\hat N^*\Om_\tau(\l)\hat N+\hat C, \qquad \hat C=\hat C^*,
\end {equation*}
where $\hat N$ is the right inverse operator for $N'=(-N_0\;\;N_1)$. This
implies that $m_\cP(\cd)$ is the uniformly strict part of the Nevanlinna
function  $\Om_\tau(\cd)$ and the function $\Om_\tau(\cd)$ is uniformly strict
 if and only if $m=\hat n=2n$ and the
$(2n\times 2n)$-matrix $(N_0\;\; N_1)$ is invertible.

In the final part of the paper we consider some questions of the eigenfunction
expansion. Namely let $\f(t,\l)=(\f_1(t,\l)\;\;\f_2(t,\l)\;\dots\; \f_d(t,\l))$
be a system of $d=d_\f$ linearly independent solutions of the equation $l[y]-\l
y=0$ with the constant initial data $\f^{(j)}(0,\l)\equiv \f_j, \; j\in
\{0,1\}$. Recall that a ($d \times d$)-matrix distribution $\S
(s)=\S_{\tau,\f}(s)\;(s\in\bR)$ is called a spectral function of the boundary
problem \eqref{1.5}-\eqref{1.8} corresponding to the solution $\f(\cd,\l)$ if
for each  function $f\in\gH$ with compact support the Fourier transform
\begin {equation*}
 g_f(s)=\int_0^b \f^\top(t,s) f(t)\,dt.
\end{equation*}
satisfies the equality
\begin {equation*}
((F_\tau(\b)-F_\tau(\a))f,f)_\gH=\int_{[\a,\b)} (d\S_{\tau,\f}(s)g_f(s),
g_f(s)), \quad [\a,\b)\subset\bR.
\end{equation*}
(here $F_\tau(\cd)$ is the spectral function of $L_0$). As is known
\cite{DunSch,Nai,Sht57} in the case $d_\f=2n$ there exists a unique spectral
function $\S_{\tau,\f}(\cd)$ of the problem \eqref{1.5}-\eqref{1.8}. At the
same time for simplification of calculations it is important to make $d_\f$ as
small as possible \cite[ch. 13.5]{DunSch}. Therefore the natural problem seems
to be a description of all spectral functions $\S_{\tau,\f}(\cd)$ with the
minimally possible value of $d_\f$ (we denote this value by $d_{min}$ and we
call the corresponding spectral function minimal). It turns out that the
complete solution of this problem is based on the introduced concept of the
$m$-function $m_\cP(\cd)$. Namely the following theorem holds.
\begin{theorem}\label{th1.1}
Let $\cP=\{C_0(\l),C_1(\l)\}$ be a Nevanlinna pair \eqref{1.4}, \eqref{1.4a}
and let $\f_N(t,\l)=(\f_1(t,\l)\;\;\f_2(t,\l)\;\dots\; \f_{\hat n}(t,\l))$ be
the $\hat n $-component linearly independent solution of the equation $l[y]-\l
y=0$ with the initial data $\f_N^{(1)}(0,\l)=-N_0^*, \;
\f_N^{(2)}(0,\l)=N_1^*$. Then:

1) there exists the unique $(\hat n\times \hat n)$-spectral function
$\S_{\cP,N}(s)$ of the problem \eqref{1.5}-\eqref{1.8} corresponding to
$\f_N(\cd,\l)$ and this function is calculated by means of the Stieltjes
formula \eqref{5.20} for the $m$-function $m_\cP(\cd)$;

2) $d_{min}=\hat n$ and the set of all minimal spectral functions
$\S_{min}(\cd)$ is given by
\begin {equation*}
\S_{min}(s)= X^* \S_{\cP,N}(s)X,
\end{equation*}
where  $X$ is an invertible $(\hat n\times \hat n)$-matrix.
\end{theorem}

Moreover we show that for a fixed pair $N=(N_0\; \; N_1)$ the set of all
spectral functions $\S_{\cP,N}(s)$ is parameterized by the Stieltjes formula
\eqref{5.20} and the following equality
\begin {equation}\label{1.12}
m_\cP(\l)=T_{N,0}(\l)+T_{N}(\l)(C_0(\l)-C_1(\l)M(\l))^{-1}C_1(\l)T_{N}^*(\ov\l),
\quad \l\in\CR,
\end{equation}
which is similar to the known Krein formula for resolvents (see for instance
\cite{DM91}). In formula \eqref{1.12} $M(\l)$ is the Weyl function \eqref{1.9}
and $T_{N,0}(\l),\;T_{N}(\l)$ are the matrix functions defined by means of
$M(\l)$ and the pair $N$. The role of a parameter in \eqref{1.12} is played by
a Nevanlinna pair $\cP=\{C_0(\l),C_1(\l)\}$ given by \eqref{1.4}, \eqref{1.4a}
with fixed $N_0,\;  N_1$ and all possible $C_{ij}'(\l)$. Note in this
connection that for a decomposing boundary problem formula \eqref{1.12} leads
to similar one  from our paper \cite{Mog07}. Moreover \eqref{1.12} implies the
known description of all Titchmarsh - Weyl functions $m(\cd)$ obtained for
quasi-regular expressions $l[y]$ by Fulton \cite{Ful77} and Khol'kin \cite
{Hol85,HolRof} (we are going to touch upon these questions elsewhere).

Finally by using Theorem \ref{th1.1} we prove the inequality $sm (\wt A)\leq
\hat n$, where $sm (\wt A)$ is the spectral multiplicity of the (exit space)
selfadjoint extension $\wt A\supset L_0$ given by the boundary conditions
\eqref{1.7}, \eqref{1.8}. This result improves the known estimate $sm (\wt
A)\leq m$ implied by simplicity of the operator $L_0$. In this connection note
that in the case $\D=[0,b]$ one can put in \eqref{1.7}, \eqref{1.8}
$\G_0'y=y^{(2)}(b), \; \G_1'y=y^{(1)}(b)$, which implies that the multiplicity
of each eigenvalue of the canonical extension $\wt A=\wt A^*$  does not exceed
$\hat n(=\text{rk}(N_0\;\;N_1))$. Hence in the case $\D=[0,b]$ the estimate $sm
(\wt A)\leq \hat n$ (for the canonical extension $\wt A$) is immediate from
\eqref{1.7}, \eqref{1.8} and discreteness of spectrum of  $\wt A$. Meanwhile,
such an estimate dose not seem to be so obvious in the case of intermediate
deficiency indices $n<m<2n$ and not decomposing boundary conditions.
\section{Preliminaries}
\subsection{Notations}
The following notations will be used throughout the paper: $\gH$, $\cH$ denote
Hilbert spaces; $[\cH_1,\cH_2]$ is the set of all bounded linear operators
defined on $\cH_1$ with values in $\cH_2$; $[\cH]:=[\cH,\cH]$; $P_\cL$ is the
orthogonal projector in $\gH$ onto the subspace $\cL\subset\gH$;
$\bC_+\,(\bC_-)$ is the upper (lower) half-plain of the complex plain.

Recall that a closed linear relation from $\cH_0$ to $\cH_1$ is a closed
subspace in $\cH_0 \oplus \cH_1$. The set of all closed linear relations from
$\cH_0$ to $\cH_1$ (from $\cH$ to $\cH$) will be denoted by  $\C (\cH_0,\cH_1)$
($\C(\cH)$). A closed linear operator $T$ from $\cH_0$ to $\cH_1$  is
identified  with its graph $\text {gr} T\in\CA$.

For a  relation $T \in \C (\cH_0,\cH_1)$ we denote by $\cD(T),\,\cR (T)$ and
$\text {Ker}T$ the domain,  range and the kernel respectively. Moreover
$T^{-1}(\in \C (\cH_1, \cH_0))$ and $T^*(\in \C (\cH_1, \cH_0))$ stands for the
inverse and adjoint  relations.

In the case $T\in\CA$ we write: $0\in \rho (T)$\ \ if\ \ $\Ker T=\{0\}$\ and\
$\cR (T)=\cH_1$, or equivalently if $T^{-1}\in [\cH_1,\cH_0]$; $0\in \hat\rho
(T)$\ \ if\ \ $\Ker T=\{0\}$\ and\ $\cR (T)$ is closed. For a linear relation
$T\in \C(\cH)$ we denote by $\rho (T)=\{\l \in \bC:\ 0\in \rho (T-\l)\}$\ and
$\hat\rho (T)=\{\l \in \bC:\ 0\in \hat\rho (T-\l)\}$ the resolvent set and the
set of regular type points of $T$ respectively.
\subsection{Holomorphic operator pairs}
Recall that a holomorphic operator function $\Phi (\cd):\bC\setminus \bR \to
[\cH]$ is called a Nevanlinna function if $Im \,\l \cd Im\, \Phi (\l)\geq 0$
and $\Phi^*(\l)=\Phi (\ov\l),\; \l\in\bC\setminus \bR$. Moreover the Nevanlinna
function $\Phi(\cd)$ is said to be uniformly  strict if $0\in\rho (Im\,
\Phi(\l))$.

Next assume that  $\Lambda$ is an open set in $\bC$,  $\cK,\cH_0,\cH_1$  are
Hilbert spaces and  $C_j(\cd):\Lambda \to [\cH_j,\cK], \; j\in \{0,1\}$ is a
pair of holomorphic operator functions (briefly a holomorphic pair). In what
follows we identify such a pair with a holomorphic operator function
\begin{equation}\label{2.5}
C(\l)=(C_0(\l)\;\;\; C_1(\l)): \cH_0\oplus\cH_1 \to\cK, \quad \l\in \Lambda.
\end{equation}
A pair \eqref{2.5}  will be called admissible if $\cR (C(\l))=\cK$ for all
$\l\in \Lambda$. In the sequel  all pairs \eqref{2.5}  are admissible unless
otherwise stated.
\begin{definition}\label{def2.0}
Two holomorphic pairs $C(\cd):\Lambda\to [\cH_0\oplus \cH_1,\cK]$ and
$C'(\cd):\Lambda\to [\cH_0\oplus \cH_1,\cK']$ are said to be equivalent if $C'
(\l)=\f (\l) C (\l), \; \l\in \Lambda$ with a holomorphic isomorphism $\f(\cd):
\Lambda\to [\cK, \cK']$.
\end{definition}
Clearly, the set of all holomorphic pairs \eqref{2.5} falls into
nonintersecting classes of equivalent pairs. Moreover such a class can be
identified with a function $\tau (\cd):\Lambda\to \CA$ given for all $\l\in
\Lambda$ by
\begin {equation}\label{2.6}
\tau (\l)=\{(C_0(\l), C_1(\l));\cK\}:=\{\{h_0,h_1\}\in \cH_0\oplus\cH_1:
C_0(\l)h_0+C_1(\l)h_1=0\}.
\end{equation}

In what follows we suppose that $\cH_0$ is a Hilbert space, $\cH_1$ is a
subspace in $\cH_0$, $\cH_2:=\cH_0\ominus\cH_1$ and  $P_j$ is the
orthoprojector in $\cH_0$ onto $\cH_j,\; j\in\{1,2\}$. With each linear
relation $\t\in\CA$ we associate a $\times$-adjoint linear relation
$\t^\times\in\CA$ defined as the set of all $\{k_0,k_1\}\in \cH_0\oplus\cH_1$
such that
\begin{equation*}
(k_1,h_0)-(k_0,h_1)+i(P_2 k_0,P_2 h_0)=0,\quad \{h_0,h_1\}\in\t.
\end{equation*}
Clearly, in the case $\cH_0=\cH_1=:\cH$ the equality $\t^\tm=\t^*$ is valid.

Next assume that $\cK_0$ is an auxiliary Hilbert space, $\cK_1$ is a subspace
in $\cK_0$ and
\begin{align}
C(\l)=(C_0(\l)\;\;C_1(\l)):\cH_0\oplus\cH_1\to\cK_0,
\;\;\l\in\bC_+\label{2.7}\\
D(\l)=(D_0(\l)\;\;D_1(\l)):\cH_0\oplus\cH_1\to\cK_1, \;\;\l\in\bC_- \label{2.8}
\end{align}
are holomorphic operator pairs with the block-matrix representations
\begin{align}
C_0(\l)=(C_{01}(\l)\;\;C_{02}(\l)):\cH_1\oplus\cH_2\to\cK_0 ,\;\label{2.8a}\\
D_0(\l)=(D_{01}(\l)\;\;D_{02}(\l)):\cH_1\oplus\cH_2\to\cK_1.\label{2.8b}
\end{align}
\begin{definition}\label{def2.1}
A Nevanlinna collection of  holomorphic operator pairs (briefly a Nevanlinna
collection) is a totality $\{C(\cd), D(\cd)\}$ of holomorphic pairs
\eqref{2.7}, \eqref{2.8} satisfying
\begin{gather}
2\,Im(C_{1}(\l)C_{01}^*(\l))+C_{02}(\l) C_{02}^*(\l) \geq 0, \;\;\; 0\in\rho
(C_0(\l)-iC_{1}(\l)P_1), \;\;
\l\in\bC_+\label{2.9}\\
2\,Im(D_1(\l)D_{01}^*(\l))+D_{02}(\l)D_{02}^*(\l)\leq 0, \;\;\; 0\in\rho
(D_{01}(\l)+iD_1(\l)), \;\;
\l\in\bC_-\label{2.10}\\
C_1(\l)D_{01}^*(\ov\l)-C_{01}(\l)D_{1}^*(\ov\l) +iC_{02}(\l)D_{02}^* (\ov\l)=0,
\;\;\; \l\in\bC_+.\label{2.11}
\end{gather}

A Nevanlinna collection \eqref{2.7}, \eqref{2.8} is said to be  constant if
$\cK_0=\cK_1=:\cK$ and $C_j(\l)=D_j(z) \equiv C_j, \; j\in \{0,1\}$ for all
$\l\in\bC_+,\; z\in\bC_- $.
\end{definition}
Clearly, a constant Nevanlinna collection can be regarded as an operator pair
\begin{equation}\label{2.12}
C=(C_0 \;\;\; C_1):\cH_0\oplus\cH_1\to\cK
\end{equation}
with the block-matrix representation $C_0=(C_{01}\;\;C_{02}): \cH_1\oplus \cH_2
\to\cK $ satisfying
\begin {equation}\label{2.12a}
2\,Im(C_{1}C_{01}^*)+C_{02}C_{02}^*=0, \quad 0\in \rho (C_0-iC_{1} P_1), \quad
0\in \rho (C_{01}+iC_{1}).
\end{equation}
This and Proposition 3.4 in \cite{Mog06.1} imply that the equality
\begin{equation}\label{2.13}
\t=\{(C_0,C_1);\cK\} :=\{\{h_0,h_1\}\in \cH_0\oplus\cH_1: C_0 h_0+C_1 h_1=0\}
\end{equation}
define a linear relation $\t\in \CA$ such that $(-\t)^\tm=- \t$. Moreover a
constant Nevanlinna collection exists if and only if $\dim\cH_1=
\dim\cH_0(=\dim\cK)$.
\begin{definition}\label{def2.2}
A collection $\pair$ of two functions $\tau_+(\cd):\bC_+\to \CA$ and
$\tau_-(\cd):\bC_-\to \CA$ is said to be of  the class $\RH$ if for all
$\l\in\bC_+$ and $z\in\bC_-$ it admits the representation
\begin {equation}\label{2.14}
\tau_+(\l)=\{(C_0(\l),C_1(\l));\cK_0\},\quad
\tau_-(z)=\{(D_0(z),D_1(z));\cK_1\}
\end{equation}
with a Nevanlinna collection $\{C(\cd), D(\cd)\}$ (see \eqref{2.6}).

A collection $\pair\in\RH$ belongs to the class $\RZ$ if it admits the
representation $\tau_\pm(\l)=\{(C_0,C_1); \cK\}=\t, \; \l\in\bC_\pm$ with a
constant Nevanlinna collection (operator pair) \eqref{2.12}.
\end{definition}
It follows from Definition \ref{def2.2} that a collection $\pair\in\RH$ can be
regarded as a collection of two equivalence classes of holomorphic pairs
\eqref{2.7} and \eqref{2.8} satisfying \eqref{2.9}--\eqref{2.11}. Moreover
according to \cite{Mog06.1}
\begin {equation}\label{2.14a}
Im \,\l\cd(2Im (h_1,h_0)-||P_2 h_0||^2)\geq 0, \quad \{h_0,h_1\}\in \tau_\pm
(\l)
\end{equation}
and $-\tau_+(\l)=(-\tau_- (\ov\l))^\tm,\;\l\in\bC_+$ for any collection
$\pair\in\RH$.
\begin{remark}\label{rem2.3}
1)Clearly a Nevanlinna collection \eqref{2.7}, \eqref{2.8} satisfies the
equalities
\begin {equation}\label{2.14b}
\dim \cH_0=\dim\cK_0, \quad \;\dim \cH_1=\dim\cK_1.
\end{equation}
Therefore   the representation \eqref{2.14} with $\cK_0=\cK_1=:\cK$ is possible
if and only if $\dim\cH_1=\dim\cH_0$, in which case the corresponding
Nevanlinna collection \eqref{2.7}, \eqref{2.8} can be regarded as the unique
holomorphic operator pair defined on $\bC_+\cup\bC_-$.

2) In the case $\cH_1=\cH_0=:\cH$ the class $\wt R(\cH):=\wt R (\cH,\cH)$
coincides with the known class of Nevanlinna functions
$\tau(\cd):\bC\setminus\bR\to\CG$ (see for instance \cite{DHM09}) and
\eqref{2.14} takes the form
\begin {equation}\label{2.15}
\tau(\l)=\{(C_0(\l),C_1(\l));\cK\}, \quad \l\in\bC\setminus\bR,
\end{equation}
where $C(\l)=(C_0(\l)\;\;\;C_1(\l)):\cH\oplus\cH\to \cK $ is a holomorphic
Nevanlinna pair. Moreover a constant Nevanlinna pair can be identified by means
of \eqref{2.12} and \eqref{2.13} with  a selfadjoint linear relation (operator
pair) $ \t=\t^*\in \C (\cH)$.
\end{remark}
\subsection{Boundary triplet and the Weyl function}
Let  $A$ be a closed densely defined symmetric  operator in $\gH$ with the
deficiency indices $n_\pm (A):=\dim \gN_\l(A), \; \l\in\bC_\pm$.
\begin{definition}\cite{Mog06.2}\label{def2.4}
A collection $\Pi=\bta$, where $\cH_0$ is a Hilbert space, $\cH_1$ is a
subspace in $\cH_0$ and $\G_j:\cD (A^*)\to \cH_j,\; j\in\{0,1\}$ are linear
maps , is called a $D$-boundary triplet (or briefly a $D$-triplet) for $A^*$,
if the map $\G=(\G_0\;\;\G_1)^\top :\cD (A^*)\to \cH_0\oplus\cH_1 $ is
surjective  and the following Green's identity holds
\begin {equation*}
(A^*f,g)-(f,A^*g)=(\G_1 f,\G_0 g)- (\G_0 f,\G_1 g)+i(P_2\G_0 f,P_2\G_0 g),
\quad f,g \in \cD (A^*)
\end{equation*}
(here as before $P_2$ is the orthoprojector in $\cH_0$ onto
$\cH_2=\cH_0\ominus\cH_1$).
\end{definition}
As was shown in \cite{Mog06.2} $$\dim \cH_1=n_-(A)\leq n_+(A)=\dim \cH_0$$ for
each $D$-triplet $\{\cH_0\oplus\cH_1,\G_0,\G_1\}$.  Moreover the equalities
\begin {equation}\label{2.16}
\cD (A_0):=\Ker \G_0=\{f\in\cD (A^*):\G_0 f=0\}, \qquad A_0=A^* \up \cD (A_0)
\end{equation}
define the maximal symmetric extension $A_0$ of $A$ with $n_-(A_0)=0$.

It turns out that for every $\l\in\bC_+\;(z\in\bC_-)$ the map $\G_0\up \gN_\l
(A)\;(P_1\G_0\up \gN_z (A))$  is an isomorphism. This makes it possible to
introduce the operator functions ($\g$-fields) $\g_+(\cdot):\Bbb
C_+\to[\cH_0,\gH], \; \; \g_-(\cdot):\Bbb C_-\to[\cH_1,\gH]$ and the Weyl
functions $M_+(\cdot):\bC_+\to [\cH_0,\cH_1], \;\; M_-(\cdot):\bC_-\to
[\cH_1,\cH_0]$ by
\begin{gather}
\g_+ (\l)=(\G_0\up\gN_\l (A))^{-1}, \;\;\;\l\in\Bbb C_+;\quad
\g_- (z)=(P_1\G_0\up\gN_z (A))^{-1}, \;\;\; z\in\Bbb C_-,\label{2.17}\\
\G_1 \up \gN_\l (A)=M_+(\l)\G_0 \up \gN_\l (A),\quad \l\in\bC_+,\label{2.18}
\\
(\G_1+iP_2\G_0)\up \gN_z (A)=M_-(z)P_1\G_0 \up \gN_z (A),\quad z\in
\bC_-.\label{2.19}
\end{gather}
According to \cite{Mog06.2} all functions $\g_\pm$ and $M_\pm $ are holomorphic
on their domains and $M_+^*(\l)=M_-(\ov \l), \;\l\in\bC_+$. Moreover the block
matrix representations
\begin{gather}
M_+(\l)=(M(\l)\;\;N_+(\l)):\cH_1\oplus\cH_2\to \cH_1,\quad \l\in\bC_+
\label{2.19a}\\
M_-(z)=(M(z)\;\;N_-(z))^\top :\cH_1\to \cH_1\oplus\cH_2,\quad
z\in\bC_-\label{2.19b}
\end{gather}
generate the uniformly strict Nevanlinna function $M(\cd):\bC\setminus\bR\to
[\cH_1]$.
\begin{proposition}\label{pr2.4a}
Let $A$ be a densely defined symmetric operator in $\gH$, let $\Pi=\bta$ be a
$D$-triplet for $A^*$ and let $M_+(\cd)$ be the corresponding Weyl function.
Then for each collection $\pair\in\RH$ the following equalities hold
\begin{gather}
s-\lim_{y\to +\infty} P_1(\tau_+(i\,y)+M_+(i\,y))^{-1}/y=0,\label{2.20}\\
s-\lim_{y\to +\infty}(M(i\,y)-M_+(i\,y)
(\tau_+(i\,y)+M_+(i\,y))^{-1}M(i\,y))/y=0.\label{2.21}
\end{gather}
\end{proposition}
\begin{remark}\label{rem2.5}
If a $D$-triplet $\Pi=\bta$ satisfies the relation
$\cH_0=\cH_1:=\cH\;(\Leftrightarrow A_0=A_0^*)$, then it is a boundary triplet.
More precisely this means that the collection $\Pi=\{\cH,\G_0,\G_1\}$ is a
boundary triplet (boundary value space) for $A^*$ in the sense of
\cite{GorGor}. In this  case the relations
\begin {equation}\label{2.28}
\g(\l)=(\G_0\up\gN_\l(A))^{-1}, \qquad \G_1\up\gN_\l(A)=M(\l)\G_0 \up\gN_\l(A),
\qquad \l\in\rho (A_0)
\end{equation}
define the operator functions \cite{DM91}  $\g(\cd):\rho (A_0)\to [\cH,\gH]$
(the $\g$-field) and  $M(\cd):\rho (A_0)\to [\cH]$ (the Weyl function)
associated  with the operator functions  \eqref{2.17}--\eqref{2.19} by
$\g(\l)=\g_\pm(\l)$  and $M(\l)=M_\pm(\l), \;\l\in\bC_\pm $. Observe also that
for a boundary triplet $\Pi$ Proposition \ref{pr2.4a} follows from the
$\Pi$-admissibility criterion obtained in \cite{DM00,DHM09}.
\end{remark}

\subsection{Differential operators}
Let $\D=[0,b\rangle\; (b\leq \infty)$ be an interval on the real axis (in the
case $b<\infty$ the point $b$ may or may not belong to $\Delta$), let $H$ be a
separable Hilbert space and let
\begin{equation}\label{3.1}
l[y]=\sum_{k=1}^n (-1)^k ( (p_{n-k}y^{(k)})^{(k)}-\tfrac {i}{2}
[(q_{n-k}^*y^{(k)})^{(k-1)}+(q_{n-k} y^{(k-1)})^{(k)}])+p_n y,
\end{equation}
be a differential expression of an even order $2n$ with smooth enough
operator-valued coefficients $p_k(\cd), q_k(\cd):\D\to [H]$ such that $
p_k(t)=p_k^*(t)$ and $ 0\in\rho (p_0(t))$. Denote by $y^{[k]}(\cd), \; k=0\div
2n$ the quasi-derivatives of a vector-function $y(\cd):\D\to H$, corresponding
to the expression \eqref{3.1} and let $\cD (l)$ be the set of functions
$y(\cd)$ for which this expression makes sense \cite{Nai, Rof69,HolRof}. With
every function $y\in\cD (l)$ we associate the functions $y^{(j)}(\cd):\D\to
H^n, \; j\in\{1,2\}$ and $\wt y(\cd):\D\to H^n\oplus H^n$ by setting
\begin{align}
y^{(1)}(t):=\{y^{[k-1]}(t)\}_{k=1}^n (\in H^n), \qquad y^{(2)}(t):=\{y^{[2n-k]}(t)\}_{k=1}^n (\in H^n),\label{3.2}\\
\wt y(t)=\{y^{(1)}(t),y^{(2)}(t)\} (\in H^n\oplus H^n), \qquad
t\in\D.\qquad\qquad \label{3.3}
\end{align}

Let $\cK$ be a Hilbert space and let $Y(\cd):\D\to [\cK,H]$ be an operator
solution of the differential equation
\begin{equation}\label{3.4}
l[y]-\l y=0.
\end{equation}
With each such a solution we associate the operator-functions
$Y^{(j)}(\cd):\D\to [\cK,H^n],\;$ $ j\in\{1,2\}$  and  $\wt Y(\cd):\D\to
[\cK,H^n\oplus H^n]$,
\begin{align*}
Y^{(1)}(t)=(Y(t)\;\;Y^{[1]}(t)\;\dots\; Y^{[n-1]}(t))^\top, \;\;\;  Y^{(2)}(t)=(Y^{[2n-1]}(t)\;\;Y^{[2n-2]}(t)\;\dots \; Y^{[n]}(t))^\top,\\
\wt Y(t)=(Y^{(1)}(t)\;\;\;Y^{(2)}(t))^\top : \cK \to  H^n\oplus H^n,\qquad
t\in\D,\qquad\qquad
\end{align*}
where $Y^{[k]}(\cd),\;k=0\div 2n-1$ are quasi-derivatives of $Y(\cd)$.

In what follows   $\gH(=L_2(\D; H))$ is a Hilbert space of all measurable
functions $f(\cd):\D\to H$ such that $\int _0^b ||f(t)||^2\,dt<\infty$.
Moreover  $\LH{\cK}$ stands for the set of all operator-functions $Y(\cd):\D\to
[\cK, H]$ such that $Y(t)h\in \gH$ for all $h\in\cK$.

It  is known \cite{Nai,Rof69,HolRof} that the expression \eqref{3.1} generates
the maximal operator $L$ in $\gH$, defined on the domain $\cD=\cD
(L):=\{y\in\cD (l)\cap \gH: l[y]\in\gH\}$ by $Ly=l[y], \; y\in\cD$. Moreover
the Lagrange's identity
\begin {equation}\label{3.5}
(Ly,z)_\gH -(y,Lz)_\gH=[y,z](b)-[y,z](0),\qquad y,z\in\cD
\end{equation}
holds with
\begin {equation}\label{3.6}
[y,z](t)=(y^{(1)}(t),z^{(2)}(t))_{H^n}-(y^{(2)}(t),z^{(1)}(t))_{H^n}, \quad
[y,z](b)=\lim_{t\uparrow b} [y,z](t).
\end{equation}

Let $\cD_0=\{y\in \cD: \wt y(0)= 0 \;\; \text{and} \;\; [y,z](b)=0, z\in\cD\}$
and let $L_0=L\up \cD_0$ be the minimal operator generated by the expression
\eqref{3.1}. Then $L_0$ is a closed densely defined symmetric operator in $\gH$
and $L_0^*=L$ \cite {Nai,Rof69,HolRof}. Moreover the deficiency indices $n_\pm
(L_0)$ of the operator $L_0$ are not necessarily equal.

Let $\t=\t^*\in\C (H^n)$ and let $L_\t$ be a symmetric extension of $L_0$ with
the domain $\cD(L_\t)=\{y\in\cD:\wt y(0)\in\t,\;[y,z](b)=0 \;\forall
z\in\cD\}$. According to \cite{Mog09.1} deficiency indices $n_\pm(L_\t)$ of an
operator $L_\t$ do not depend on $\t(=\t^*)$, which enables us to introduce the
deficiency indices  at the right endpoint $b$ by $n_{b\pm}:=n_\pm(L_\t)$.
\subsection{Decomposing boundary triplets}
Assume that $\cH'_1$ is a subspace in a Hilbert space $\cH'_0$,
$\cH_2':=\cH'_0\ominus\cH_1'$, $\G_0':\cD\to \cH'_0$ and $\G_1':\cD\to \cH'_1$
are linear maps and  $P'_j$ is the orthoprojector in $\cH'_0$ onto $\cH'_j, \;
j\in\{1,2\}$. Moreover let $\cH_0=H^n\oplus \cH'_0, \;\cH_1=H^n\oplus \cH'_1$
and let $\G_j:\cD\to \cH_j,\;j\in\{0,1\} $ be linear maps given for all
$y\in\cD$ by
\begin {equation}\label{3.10}
\G_0 y=\{y^{(2)}(0), \G'_0y\}\,(\in H^n\oplus \cH'_0), \;\;\; \G_1
y=\{-y^{(1)}(0), \G'_1y\}\,(\in H^n\oplus \cH'_1) .
\end{equation}
\begin{definition}\label{def3.2}\cite{Mog09.1}
A collection $\Pi=\bta$, where $\G_0$ and $\G_1$ are linear maps \eqref{3.10},
is said to be a decomposing $D$-boundary triplet (briefly a decomposing
$D$-triplet) for $L$ if the map
$\G'=(\G'_0\;\;\G'_1)^\top:\cD\to\cH_0'\oplus\cH'_1$ is surjective and the
following identiy holds
\begin{equation}\label{3.11}
[y,z](b)=(\G'_1 y,\G'_0 z)- (\G'_0 y,\G'_1 z)+i(P'_2\G'_0 y,P'_2\G'_0 z), \quad
y,z \in \cD.
\end{equation}
\end{definition}
In the case $\cH'_0=\cH'_1=:\cH'\;(\iff \cH_0=\cH_1=:\cH)$ a decomposing
$D$-triplet $\Pi=\bt$ is called a decomposing boundary triplet for $L$. For
such a triplet the identity \eqref{3.11} takes the form
\begin {equation}\label{3.12}
[y,z](b)=(\G'_1 y,\G'_0 z)- (\G'_0 y,\G'_1 z), \quad y,z \in \cD.
\end{equation}
As was shown in \cite{Mog09.1}, Lemma 3.4 a decomposing $D$-triplet (a
decomposing boundary triplet) for $L$ is a $D$-triplet (a boundary triplet) in
the sense of Definition \ref{def2.4} and Remark \ref{rem2.5}. Moreover  a
decomposing $D$-triplet (boundary triplet) for $L$  exists if and only if
$n_{b-}\leq n_{b+}$ (respectively, $n_{b-}= n_{b+}$), in which case
\begin {equation}\label{3.13}
\dim\cH_1'=n_{b-}\leq n_{b+}=\dim\cH_0', \qquad \dim\cH_1=n_-(L_0)\leq n_+(L_0)
=\dim \cH_0
\end{equation}
(respectively, $ n_{b-}=n_{b+}=\dim\cH'$ and $ n_-(L_0)=n_+(L_0)=\dim \cH $).
Therefore in the sequel we suppose (without loss of generality)   that
$n_{b-}\leq n_{b+}$ and, consequently, $n_-(L_0)\leq n_+(L_0)$.
\begin{proposition}\label{pr3.2a}\cite{Mog09.1}
Let $\bta$ be a decomposing $D$-triplet \eqref{3.10} for $L$ and let $\g_\pm
(\cd)$ be the corresponding $\g$-fields \eqref{2.17}. Then:

1)  For each $\l\in\bC_+\;(z\in\bC_-)$ there exists a  unique operator function
$Z_+(\cd,\l)\in\LH{\cH_0}\;(Z_-(\cd,z)\in\LH{\cH_1})$, satisfying \eqref{3.4}
and the boundary condition $\G_0(Z_+(t,\l)h_0)=h_0, \; h_0\in \cH_0$ (resp.
$P_1\G_0(Z_-(t,z)h_1)=h_1, \; h_1\in \cH_1$). If
\begin{align}
Z_+(t,\l)=(v_0(t,\l)\;\;u_+(t,\l)):H^n\oplus\cH'_0\to H, \quad \l \in \bC_+,
\label{3.17}\\
Z_-(t,z)=(v_0(t,z)\;\;u_-(t,z)):H^n\oplus\cH'_1\to  H, \quad
z\in\bC_-\label{3.18}
\end{align}
are the block representations of $Z_+(\cd,\l)$ and $Z_-(\cd,z)$, then the above
boundary condition can be represented as
\begin{gather*}
v_0^{(2)}(0,\mu)=I_{H^n}\;\;(\mu\in\CR); \;\;  \G'_0(v_0(t,\l)\hat h)=0,\;\;
P'_1\G'_0(v_0(t,z)\hat h)=0,\;\; \hat h\in H^n\\
u_+^{(2)}(0,\l)=0; \quad \G'_0 (u_+(t,\l) h'_0)=h_0',\quad \l\in\bC_+, \;\;\;
 h'_0\in\cH'_0; \\
 u_-^{(2)}(0,z)=0, \quad P_1'\G'_0(u_-(t,z) h'_1)=h_1',\quad z\in\bC_-,\;\;\;
 h'_1\in\cH'_1.
\end{gather*}

2) for all $\l\in\bC_+$ and $z\in\bC_-$ the following equalities hold
\begin {equation}\label{3.18a}
(\g_+(\l) h_0)(t)=Z_+(t,\l) h_0,\;\;  h_0\in \cH_0\quad (\g_-(z)
h_1)(t)=Z_-(t,z) h_1, \;\;  h_1\in \cH_1.
\end{equation}
\end{proposition}
Next assume that $M_\pm(\cd)$ are the  Weyl functions \eqref{2.18},
\eqref{2.19} corresponding to the $D$-triplet $\Pi$ and let
\begin{gather}
M_+(\l)=\begin{pmatrix} m(\l) & M_{2+}(\l)\cr M_{3+}(\l) &
M_{4+}(\l)\end{pmatrix}:
H^n\oplus\cH'_0 \to H^n\oplus \cH'_1, \quad \l\in\bC_+\label{3.21}\\
M_-(z)=\begin{pmatrix} m(z) & M_{2-}(z)\cr M_{3-}(z) & M_{4-}(z)\end{pmatrix}:
H^n\oplus\cH'_1 \to H^n\oplus \cH'_0, \quad z\in\bC_-\label{3.22}
\end{gather}
be their block representations. As was proved in \cite [Theorem 3.12]{Mog09.1}
all the entries in \eqref{2.18} and \eqref{2.19} can be defined immediately in
terms of boundary values of the functions $v_0(\cd,\l)$ and $u_\pm(\cd,\l)$. In
particular formulas \eqref{3.21} and \eqref{3.22} generate the uniformly strict
Nevanlinna function $m(\l)=-v_0^{(1)}(0,\l)$, which we called in \cite{Mog09.1}
the $m$-function.
\subsection{Generalized resolvents and characteristic matrices}
Let $\wt A\supset L_0$ be an exit space selfadjoint  extension of the operator
$L_0$ acting in the Hilbert space $\wt \gH\supset\gH$ and let $\wt E (t)$ be
the orthogonal spectral function of the operator $\wt A$. Recall that the
operator functions $\R=P_\gH (\wt A- \l)^{-1}\vert \gH, \; \l \in\bC
\setminus\bR $ and $F(t)=P_\gH \wt E (t)\vert \gH$ are called  generalized
resolvent and  spectral function of the operator $L_0$ respectively. In the
sequel we suppose that the spectral function  $\wt E (t)$ (or equivalently  the
extension $\wt A$) is minimal, which means that $\text {span} \{\gH, \wt E(t)
\gH : t\in\bR\}=\wt\gH$.

Let $Y_0(\cd,\l):\D\to [H^n\oplus H^n, H]$ be the "canonical" operator solution
of the equation \eqref{3.4} with the initial data $\wt Y_0 (0,\l)=I_{H^n\oplus
H^n}$ and let
\begin {equation}\label{3.30a}
J_{H^n}:=\begin{pmatrix} 0 & -I_{H^n} \cr I_{H^n} & 0
\end{pmatrix}: H^n\oplus H^n\to  H^n\oplus H^n.
\end{equation}
 According to \cite{Sht57,Bru74} the
generalized resolvent $\R$ admits the representation
\begin {equation}\label{3.30}
(\R f)(x)=\smallint\limits_0^b  G (x,t,\l)f(t)\, dt:=\lim_{\eta \uparrow
b}\smallint\limits_0^\eta   G (x,t,\l)f(t)\, dt,\;\; f=f(\cd)\in\gH
\end{equation}
with the Green function $G(\cd,\cd,\l):\D\times\D\to [H]$ given by
\begin {equation}\label{3.31}
G(x,t,\l)=Y_0(x,\l)(\Om(\l)+\tfrac 1 2 \,
\emph{sgn}(t-x)J_{H^n})Y_0^*(t,\ov\l), \quad \l\in\bC\setminus\bR.
\end{equation}
Here $\Om(\l)(\in [H^n\oplus H^n])$ is a Nevanlinna operator function, which is
called a characteristic matrix of the generalized resolvent $\R$ \cite{Sht57}.

Next assume that $\Pi=\bta$ is a decomposing $D$-triplet \eqref{3.10} for $L$
and  $\pair\in\RH$ is a collection of holomorphic  pairs \eqref{2.14} with the
block representations
\begin{gather}
C_0(\l)=(\hat C_{0}(\l)\;\;C'_{0}(\l)): H^n\oplus\cH_0'\to\cK_0, \qquad
\qquad\qquad \qquad\qquad \qquad\qquad  \label{3.25}\\
\qquad \qquad\qquad \qquad  C_1(\l)=(\hat C_{1}
(\l)\;\;C'_1(\l)):H^n\oplus \cH_1'\to\cK_0 , \quad \l\in\bC_+\nonumber\\
 D_0(\l)=(\hat D_{0}(\l)\;\;D'_{0}(\l)):H^n\oplus\cH_0'\to\cK_1,
\qquad \qquad\qquad \qquad\qquad \qquad\qquad\label{3.26}\\
\qquad \qquad\qquad \qquad  D_1(\l)=(\hat D_{1}(\l)\;\;D'_1(\l)):H^n\oplus
\cH_1'\to\cK_1, \quad \l\in\bC_-\nonumber
\end{gather}
For a given function $f\in\gH$ consider the boundary value problem
\begin{gather}
l[y]-\l y=f \label{3.27}\\
\hat C_0(\l) y^{(2)}(0)+\hat C_1(\l) y^{(1)}(0)+ C'_0(\l) \G'_0 y-C'_1(\l)
\G'_1 y=0, \quad \l\in\bC_+\label{3.28}\\
\hat D_0(\l) y^{(2)}(0)+\hat D_1(\l) y^{(1)}(0)+ D'_0(\l) \G'_0 y-D'_1(\l)
\G'_1 y=0, \quad \l\in\bC_-.\label{3.29}
\end{gather}
In view of \eqref{3.25} and \eqref{3.26}  the conditions \eqref{3.28} and
\eqref{3.29} can be written as
\begin{gather}\label{3.29a}
C_0(\l)\G_0 y-C_1(\l)\G_1 y=0, \;\; \l\in\bC_+;\;\;\; D_0(\l)\G_0 y-D_1(\l)\G_1
y=0, \;\; \l\in\bC_-.
\end{gather}
A function $y(\cd,\cd):\D\tm (\CR) \to H$ is  called a solution of the boundary
problem \eqref{3.27}--\eqref{3.29} if for each $\l\in\CR$ the function
$y(\cd,\l)$ belongs to $\cD$ and satisfies the equation \eqref{3.27} and the
boundary conditions \eqref{3.28}, \eqref{3.29}.
\begin{theorem}\label{th3.5}\cite{Mog10}
 Let  $\pair\in\RH$ be a collection given by \eqref{2.14} and \eqref{3.25},
\eqref{3.26} and let $\wt\Om_{\tau +}(\l)$  and $\wt\Om_{\tau -}(\l)$ be the
operator function defined by
\begin {gather}
\wt\Om_{\tau +}(\l)=\begin{pmatrix}\wt\om_{1+}(\l) & \wt\om_{2+}(\l) \cr
\wt\om_{3+}(\l) & \wt\om_{4+}(\l) \end{pmatrix}:\cH_0\oplus\cH_1\to
\cH_1\oplus\cH_0, \quad \l\in\bC_+,\label{3.19}\\
\wt\om_{1+}(\l)=M_+(\l)-M_+(\l)(\tau_+(\l) +M_+(\l))^{-1}M_+(\l)\label{3.19.1}\\
\wt\om_{2+}(\l)= -\tfrac 1 2 I_{\cH_1} +M_+(\l) (\tau_+(\l) +M_+(\l))^{-1}
\label{3.19.2}\\
\wt\om_{3+}(\l)=-\tfrac 1 2 I_{\cH_0} +(\tau_+(\l) +M_+(\l))^{-1}M_+(\l\label{3.19.3})\\
\wt\om_{4+}(\l)=-(\tau_+(\l) +M_+(\l))^{-1},\label{3.19.4}\\
\wt\Om_{\tau_-}(\l)=(\wt\Om_{\tau +}(\ov\l))^*,\quad \l\in\bC_-.\label{3.19.5}
\end{gather}
 Then:

1) for each $f\in\gH$ the boundary problem \eqref{3.27}--\eqref{3.29} has the
unique solution $y(t,\l)=y_f(t,\l)$ and the equality $(R(\l)f)(t)=y_f(t,\l),\;
f\in\gH,\;\l\in\bC\setminus\bR$ defines a generalized resolvent  $\R:= R_\tau
(\l)$ of the minimal operator $L_0$;

2) the characteristic matrix of the generalized resolvent  $ R_\tau (\l)$ is
\begin {equation}\label{3.24}
\Om(\l)=\Om_\tau (\l):=P_{H^n\oplus H^n}\,\wt \Om_{\tau +}(\l)\up H^n\oplus
H^n,\quad \l\in\bC_+.
\end{equation}
Conversely for each generalized resolvent  $\R$ there exists the unique
$\tau\in\RH$ such that $\R=R_\tau (\l)$. Moreover $R_\tau (\l)$ is a canonical
resolvent if and only if $\tau\in \RZ$.
\end{theorem}
\begin{proposition}\label{pr3.6}
Assume that $\pair\in\RH$ ia a collection given by \eqref{3.25} and
\eqref{3.26}, $\wt D_1(\l)(\in[\cH_1,\cK_-] )$ and $\wt
D_0(\l)(\in[\cH_0,\cK_-] )$ are defined by
\begin {equation}\label{3.33}
\wt D_1(\l):=D_0(\l)\up \cH_1, \;\;\;\; \wt D_0(\l)=D_1(\l) P_1+iD_0(\l)P_2,
\;\;\;\;\l\in\bC_-,
\end{equation}
and $D_{01}(\l)$, $D_{02}(\l)$ are entries of the block representation
\eqref{2.8b}. Moreover let $\g_\tau (\l)(\in [H^n\oplus H^n,\gH])$ and $
\wt\a(\l)(\in [\cK_-,H^n\oplus H^n]) $ be operator functions given by the block
matrix representations
\begin{gather}
\g_\tau (\l)=\g_+(\l)(C_0(\l)-C_1(\l)M_+(\l))^{-1}(-\hat C_0(\l)\;:\; \hat
C_1(\l)), \quad \l\in\bC_+\label{3.34}\\
\wt\a(\l)=\begin{pmatrix}-P_{H^n}M_-(\l) \cr P_{H^n} \end{pmatrix} (\wt
D_1(\l)-\wt D_0(\l)M_-(\l))^{-1},\quad \l\in\bC_-.\label{3.35}
\end{gather}
Then the corresponding characteristic matrix $\Om_\tau(\cd)$  satisfies the
identity
\begin{gather}
\Om_\tau(\mu)-\Om_\tau^*(\l)=(\mu-\ov\l)\g_\tau^*(\l)\g_\tau(\mu)-
  \wt\a(\ov\l)(D_{1}(\ov\l)D_{01}^*(\ov\mu)-D_{01}(\ov\l)D_1^*(\ov\mu)+
 \label{3.36}\\
+i D_{02}(\ov\l)D_{02}^*(\ov\mu))\wt\a^*(\ov\mu), \quad\; \mu,\l\in
\bC_+.\nonumber
\end{gather}
 Moreover the following equality holds
\begin {equation}\label{3.37}
s-\lim_{y\to\infty} \Om_\tau (i \,y)/y=0
\end{equation}
\end{proposition}
\begin{proof}
The identity \eqref{3.36} was proved in \cite{Mog10}. To prove \eqref{3.37}
assume  that
\begin {equation}\label{3.38}
\Om_\tau (\l)=\begin{pmatrix} \om_1(\l) & \om_2(\l) \cr \om_3(\l) & \om_4(\l)
\end{pmatrix}:H^n\oplus H^n\to H^n\oplus H^n, \qquad \l\in\CR
\end{equation}
is the block matrix representation of $\Om_\tau (\l)$. Then by \eqref{3.19.1}
and \eqref{3.19.4}
\begin{gather*}
\om_1(\l)=P_{H^n}(\wt\om_{1+}(\l)\up \cH_1)\up H^n=\\
P_{H^n}(M(\l)-
M_+(\l)(\tau_+(\l) +
M_+(\l))^{-1}M(\l))\up H^n,\\
\om_4(\l)=P_{H^n}(P_1 \,\wt\om_{4+}(\l))\up H^n=-P_{H^n}(P_1 (\tau_+(\l)
+M_+(\l))^{-1})\up H^n,\quad \l \in\bC_+,
\end{gather*}
which in view of \eqref{2.20} and \eqref{2.21} gives
\begin {equation*}
s-\lim_{y\to\infty}\om_1(i\,y)/y=s-\lim_{y\to\infty}\om_4(i\,y)/y=0.
\end{equation*}
This and the representation \eqref{3.38} proves the equality \eqref{3.37}.
\end{proof}
\begin{remark}\label{rem3.5a}
It follows from Theorem \ref{th3.5} that the boundary problem
\eqref{3.27}-\eqref{3.29} gives a parameterization  of all generalized
resolvents $R(\l)=R_\tau(\l)$ and characteristic matrices $\Om
(\l)=\Om_\tau(\l)$ by means of the Nevanlinna boundary parameter $\tau$.
Moreover since a spectral function $F(t)$ is uniquely defined by the
corresponding generalized resolvent $\R$, one obtains the parameterization
$F(t)=F_\tau(t)$ of all spectral functions of the operator $L_0$ by means of
the same boundary parameter $\tau$.
\end{remark}
\section{$m$-functions and characteristic matrices}
\subsection{Quasi-constant and $N$-triangular Nevanlinna collections}
Let $\Pi=\bta$ be a decomposing $D$-triplet \eqref{3.10}  for $L$ (with
$\cH_j=H^n\oplus \cH_j',\; j\in\{0,1\}$). A Nevanlinna collection $\CD$ defined
by \eqref{2.7}, \eqref{2.8} and the block representations \eqref{3.25},
\eqref{3.26} will be called quasi-constant if $\hat C_j(\l)=\hat D_j(z)\equiv
\hat C_j(\in [H^n,\cK_1]), \; j\in \{0,1\}$  for all $\l\in\bC_+$ and
$z\in\bC_-$ (such a definition is correct, since $\cK_1\subset\cK_0$). Clearly,
each constant  pair $\t(=\t^*)=\{(C_0,C_1); \cK\}$ is quasi-constant.

Next assume that
\begin {equation}\label{4.1}
N=(N_0\;\;N_1):H^n\oplus H^n\to \hat \cK
\end{equation}
is an admissible operator pair (that is $\cR (N)=\hat \cK$) and let $\t_N\in \C
(H^n) $ be a linear relation given by $\t_N=\{(N_0,N_1);\hat\cK\}$. The
operator pair \eqref{4.1} will be called symmetric (selfadjoint) if the linear
relation $\t_N$ is symmetric (selfadjoint).
\begin{definition}\label{def4.1}
 A Nevanlinna  collection $\{C(\cd),
D(\cd)\}$ defined by \eqref{2.7}, \eqref{2.8} will be called $N$-triangular if
there exist a Hilbert space $\cK'_0$ and a subspace $\cK'_1\subset\cK'_0$ such
that $\cK_j=\hat\cK\oplus \cK'_j, \;  j\in \{0,1\}$ and the following block
representations hold
\begin{gather}
C_0(\l)=\begin{pmatrix} N_0 & C'_{01}(\l)\cr 0 & C'_{02}(\l)
\end{pmatrix} : H^n\oplus \cH'_0\to \hat\cK\oplus \cK'_0, \;\quad \l\in\bC_+
\label{4.2}\\
C_1(\l)=\begin{pmatrix} N_1 & C'_{11}(\l)\cr 0 & C'_{12}(\l)
\end{pmatrix} : H^n\oplus \cH'_1\to \hat\cK\oplus \cK'_0, \;\quad \l\in\bC_+
\label{4.3}\\
D_0(\l)=\begin{pmatrix} N_0 & D'_{01}(\l)\cr 0 & D'_{02}(\l)
\end{pmatrix} : H^n\oplus \cH'_0\to \hat\cK\oplus \cK'_1, \;\quad \l\in\bC_-
\label{4.4}\\
D_1(\l)=\begin{pmatrix} N_1 & D'_{11}(\l)\cr 0 & D'_{12}(\l)
\end{pmatrix} : H^n\oplus \cH'_1\to \hat\cK\oplus \cK'_1, \;\quad \l\in\bC_-
\label{4.5}
\end{gather}
\end{definition}
A constant  $N$-triangular collection can be regarded as an operator pair
\begin {equation}\label{4.6}
C=(C_0\;\;C_1):(H^n\oplus\cH'_0)\oplus(H^n\oplus\cH'_1) \to \hat\cK\oplus \cK'
\end{equation}
defined by the block matrix representations
\begin {equation}\label{4.7}
C_0=\begin{pmatrix} N_0 & C'_{01}\cr 0 & C'_{02}
\end{pmatrix} : H^n\oplus \cH'_0\to \underbrace{\hat\cK\oplus \cK'}
_{ \cK} , \quad C_1=\begin{pmatrix} N_1 & C'_{11}\cr 0 & C'_{12}
\end{pmatrix} : H^n\oplus \cH'_1\to \underbrace{\hat\cK\oplus \cK'}_{ \cK}
\end{equation}
and satisfying the relations \eqref{2.12a}.

Assume now that $\CD$ is a quasi-constant Nevanlinna  collection defined by
\eqref{2.7}, \eqref{2.8} and \eqref{3.25},\eqref{3.26} and let $\hat\cK(\subset
\cK_1)$ be the range of the operator
\begin {equation}\label{4.13}
\hat C:=C(\l)\up H^n\oplus H^n=(\hat C_0 \;\;\hat C_1):H^n\oplus H^n\to \cK_1.
\end{equation}
It is clear that the collection $\CD$ is $N$-triangular with some $N$ if and
only if $\hat\cK$ is closed, in which case $N_j=\hat C_j(\in [H^n,\hat\cK]),
\;j\in \{0,1\}$ (here $\hat C_j$ is considered as acting from $H^n$ to
$\hat\cK$). In this connection the following proposition holds.
\begin{proposition}\label{pr4.4a}
If $n_{b+}<\infty$ (in particular, $\dim H<\infty$), then each quasi-constant
Nevanlinna collection is $N$-triangular.
\end{proposition}
\begin{proof}
Since the operator pair \eqref{2.7} is admissible, it follows that $\cR
(C(\l))=\cK_0$ and, therefore, the range of the operator $C^*(\l)$ is a closed
subspace in $\cH_0\oplus\cH_1$. Moreover by \eqref{4.13} $\hat C^*=P_{H^n\oplus
H^n}C^*(\l)$ and, consequently,
\begin {equation}\label{4.13.1}
 \cR (\hat C^*)=P_{H^n\oplus H^n}\cR
(C^*(\l)).
\end{equation}
 Since in view of \eqref{3.13}  $\codim (H^n\oplus H^n)=\dim
(\cH'_0\oplus\cH'_1)<\infty$, it follows from \eqref{4.13.1} that  $\cR (\hat
C^*)$ is a closed subspace in $H^n\oplus H^n$. This implies that $\hat\cK(=\cR
(\hat C))$ is also closed.
\end{proof}
\begin{remark}\label{rem4.5}
In the case $n_{b+}=\infty(\Leftrightarrow\dim\cH_0'=\infty)$ one can easy
construct a quasi-constant (and even constant) Nevanlinna collection $\CD$ with
 not closed subspace $\hat\cK$, which implies that this collection is not
$N$-triangular with any $N$. Hence the condition $n_{b+}<\infty$ in Proposition
\ref{pr4.4a} is essential.
\end{remark}

Two $N$-triangular Nevanlinna collections $\{C(\cd),D(\cd)\}$ and $\{\wt
C(\cd), \wt D(\cd)\}$ (with the same $N$) are said to be equivalent if the
operator pairs $C(\cd)$ and $\wt C(\cd)$ as well as $D(\cd)$ and $\wt D(\cd)$
are equivalent in the sense of Definition \ref{def2.0}. It is clear that for a
given operator pair $N$ (see \eqref{4.1}) the set of all $N$-triangular
Nevanlinna collections falls into nonintersecting equivalence classes. In what
follows the set of all such classes will be denoted by $TR\{\cH_0,\cH_1\}$.
Moreover we will denote by $\cP=\{C(\cd),D(\cd)\}$ both an $N$-triangular
Nevanlinna collection and the corresponding equivalence class.

\begin{definition}\label{def4.3}
A collection (the corresponding equivalence class) $\cP=\{C(\cd),D(\cd)\}\in\;$
$TR\{\cH_0,\cH_1\}$ is said to belong to the class $\TRZ$ if it admits the
representation \eqref{4.6}, \eqref{4.7} as a constant $N$-triangular
collection.
\end{definition}
In the sequel we write $\cP=\{C_0,C_1\}\in\TRZ$ identifying the collection
$\cP\in\TRZ$ and the corresponding operator pair \eqref{4.6}, \eqref{4.7}.

In the case $\cH_0=\cH_1=:\cH$ (i.e., in the case of a decomposing boundary
triplet $\Pi=\bt$) we let $TR(\cH):=TR(\cH,\cH)$ and
$TR^0(\cH):=TR^0(\cH,\cH)$.
\begin{proposition}\label{pr4.4}
Assume that $N=(N_0\;\;N_1)$ is an operator pair \eqref{4.1} and $\CD\in\TR$ is
a collection \eqref{4.2}-\eqref{4.5}. Then the pair $N$ is symmetric and
\begin {equation}\label{4.10}
n\dim H \leq \dim \hat\cK\leq n_-(L_0);
\end{equation}
\end{proposition}
\begin{proof}
Let $\tau_\pm(\l)$ be linear relations \eqref{2.14}. Then in view of
\eqref{4.2}-\eqref{4.5} $\t_N=\tau_\pm(\l)\cap (H^n\oplus H^n), \;
\l\in\bC_\pm$ and \eqref{2.14a} shows that $\t_N$ is a symmetric linear
relation. Therefore $\dim H^n\leq \codim \t_N=\dim \hat\cK$, which together
with \eqref{2.14b} and the second relation in \eqref{3.13} gives \eqref{4.10}.
\end{proof}
\subsection{Generalized resolvents and the Green function}
Let $\Pi=\bta$ be a decomposing $D$-triplet \eqref{3.10} for $L$ and let $\cP
=\CD\in\TR$ be a collection \eqref{4.2}-\eqref{4.5}. Then the corresponding
boundary problem \eqref{3.27}-\eqref{3.29} can be written as
\begin{gather}
l[y]-\l y=f \label{4.16}\\
N_0 y^{(2)}(0)+N_1 y^{(1)}(0)+C_{01}'(\l)\G'_0 y- C_{11}'(\l)\G'_1
y=0,\label{4.17}\\
C_{02}'(\l)\G'_0 y- C_{12}'(\l)\G'_1 y=0, \qquad \l\in\bC_+;\label{4.18}\\
N_0 y^{(2)}(0)+N_1 y^{(1)}(0)+D_{01}'(\l)\G'_0 y- D_{11}'(\l)\G'_1
y=0,\label{4.19}\\
D_{02}'(\l)\G'_0 y- D_{12}'(\l)\G'_1 y=0, \qquad \l\in\bC_-.\label{4.20}
\end{gather}
Moreover in the case $\cP=\{C_0,C_1\}\in\TRZ$ (see \eqref{4.6} and \eqref{4.7})
the boundary conditions \eqref{4.17}-\eqref{4.20} take the form
\begin{gather}
N_0 y^{(2)}(0)+N_1 y^{(1)}(0)+C_{01}'\G'_0 y- C_{11}'\G'_1 y=0\label{4.21}\\
C_{02}'\G'_0 y- C_{12}'\G'_1 y=0\label{4.22}
\end{gather}
The following corollary is immediate from Theorem \ref{th3.5}.
\begin{corollary}\label{cor4.7}
Let $\Pi=\bta$ be a decomposing $D$-triplet \eqref{3.10} for $L$ and let
$\cP=\CD\in\TR$ be a collection given by \eqref{4.2}-\eqref{4.5}. Then the
boundary problem \eqref{4.16}-\eqref{4.20} generates the generalized resolvent
$\R=R_\cP (\l)$ of the operator $L_0$ (in the same way as in Theorem
\ref{th3.5}). Moreover $\R$ is a canonical resolvent if and only if
$\cP\in\TRZ$, in which case the corresponding boundary conditions can be
defined by \eqref{4.21} and \eqref{4.22}.
\end{corollary}
\begin{remark}\label{rem4.7a}
Note that in view of Corollary \ref{cor4.7} the generalized resolvent $\R=
R_\cP (\l)$ can be also defined by $R_\cP (\l)=(\wt A(\l)-\l)^{-1}, \;\l\in\CR$
where $\wt A(\l)=L\up \cD (\wt A(\l))$ and $\cD (\wt A(\l))$ is the set of all
functions $y\in\cD$ satisfying the boundary conditions
\eqref{4.17}-\eqref{4.20} or, equivalently, \eqref{3.29a}.
\end{remark}
Assume that  $\cP=\CD\in\TR$ is a collection \eqref{4.2}-\eqref{4.5} and let
$\wt D_1(\l)(\in [\cH_1,\cK_1])$ and $\wt D_0(\l)(\in [\cH_0,\cK_1])$ be
defined by
\begin {equation*}
\wt D_1(\l):=D_0(\l)\up \cH_1, \;\;\;\; \wt D_0(\l)=D_1(\l) P_1+iD_0(\l)P_2,
\;\;\;\;\l\in\bC_-.
\end{equation*}
It follows from \eqref{4.4} and \eqref{4.5} that  the following block
representations hold
\begin{gather*}
\wt D_1(\l)=\begin{pmatrix} N_0 & \wt D'_{01}(\l)\cr 0 & \wt D'_{02}(\l)
\end{pmatrix} : H^n\oplus \cH'_1\to \hat\cK\oplus \cK'_1, \;\quad \l\in\bC_-
\\
\wt D_0(\l)=\begin{pmatrix} N_1 & \wt D'_{11}(\l)\cr 0 & \wt D'_{12}(\l)
\end{pmatrix} : H^n\oplus \cH'_0\to \hat\cK\oplus \cK'_1, \;\quad \l\in\bC_-.
\end{gather*}
\begin{proposition}\label{pr4.8}
Let the conditions of Corollary \ref{cor4.7} be satisfied. Then:

1) for each $\l\in\CR$ there exists the unique operator function
$v(\cd,\l)\in\LH{\hat\cK}$ satisfying the equation $l[y]-\l y=0$ and the
boundary conditions
\begin{gather}
(N_0 v^{(2)}(0,\l)+N_1 v^{(1)}(0,\l))\hat h+(C_{01}'(\l)\G'_0 -
C_{11}'\G'_1)(v(t,\l)\hat h)=\hat h \label{4.25}\\
(C_{02}'(\l)\G'_0 - C_{12}'(\l)\G'_1) (v(t,\l)\hat h)=0, \qquad \hat h\in\hat
\cK, \;\; \l\in\bC_+; \label{4.26}\\
(N_0 v^{(2)}(0,\l)+N_1 v^{(1)}(0,\l)) \hat h+ (D_{01}'(\l)\G'_0 -
D_{11}'(\l)\G'_1)(v(t,\l)\hat h)=\hat h\label{4.27}\\
(D_{02}'(\l)\G'_0 - D_{12}'(\l)\G'_1)(v(t,\l)\hat h) =0, \qquad \hat h\in\hat
\cK, \;\; \l\in\bC_-.\label{4.28}
\end{gather}
2) The  functions $v(\cd,\l)$ and $Z_\pm (\cd,\l)$ (see \eqref{3.17} and
\eqref{3.18}) are connected by
\begin {equation}\label{4.29}
v(t,\l)=\begin{cases} Z_+(t,\l)(C_0(\l)-C_1(\l)M_+(\l))^{-1}\up \hat\cK,
\;\;\;\l\in\bC_+\\
Z_-(t,\l)(\wt D_1(\l)-\wt D_0(\l)M_-(\l))^{-1}\up \hat\cK, \;\;\;\l\in\bC_-,
\end{cases}.
\end{equation}
where $M_\pm (\cd)$ are the Weyl functions \eqref{3.21} and \eqref{3.22} for
$\Pi$.

3) for each $\l\in\bC_+$ (resp. $\l\in\bC_-$) the equality $y(t)=v(t,\l)\hat h$
gives a bijective correspondence between all $\hat h\in\hat\cK$ and all
solutions $y(\cd) $ of the equation \eqref{3.4}, which belong to $\gH$ and
satisfy the boundary condition \eqref{4.18} (resp. \eqref{4.20}). Therefore the
operator function $v_\cP(\cd,\l)$ is a fundamental solution of the  boundary
problems  \eqref{3.4}, \eqref{4.18} for $\l\in\bC_+$ and  \eqref{3.4},
\eqref{4.20} for $\l\in\bC_-$(see \cite{Mog09.2,HolRof}).
\end{proposition}
\begin{proof}
1)-2) It follows from \eqref{3.10} and \eqref{4.2}-\eqref{4.5} that the
conditions \eqref{4.25}-\eqref{4.28} are equivalent to
\begin{gather}
(C_0(\l)\G_0-C_1(\l)\G_1)(v(t,\l)\hat h)=\hat h, \qquad \hat h\in\hat \cK, \;\;
\l\in\bC_+\label{4.30}\\
(D_0(\l)\G_0-D_1(\l)\G_1)(v(t,\l)\hat h)=\hat h, \qquad \hat h\in\hat \cK, \;\;
\l\in\bC_-.\label{4.31}
\end{gather}
As was shown in \cite{Mog10} $0\in\rho (C_0(\l)-C_1(\l)M_+(\l)), \; 0\in\rho
(\wt D_1(\l)-\wt D_0(\l)M_-(\l))$  and
\begin{gather}
(C_0(\l)\G_0-C_1(\l)\G_1)(Z_+(t,\l) h)=(C_0(\l)-C_1(\l)M_+(\l))h, \;\;\;\;
h\in\cH_0, \;\l\in\bC_+ \label{4.32}\\
(D_0(\l)\G_0-D_1(\l)\G_1)(Z_-(t,\l) h)=(\wt D_1(\l)-\wt D_0(\l)M_-(\l))h, \;
h\in\cH_1, \;\;\l\in\bC_-.\label{4.33}
\end{gather}
Hence the equality \eqref{4.29} correctly defines the  function
$v(\cd,\l)\in\LH{\hat\cK}$ satisfying \eqref{4.30}, \eqref{4.31} and
consequently \eqref{4.25}-\eqref{4.28}. The uniqueness of such a function
follows from the inclusion $\l\in\rho (\wt A(\l))$, where $\wt A(\l)$ is
defined in Remark \ref{rem4.7a}.

3) If $\hat h\in\hat \cK$, then by the statement 1) the function $y(t)=v(t,\l)
\hat h$ satisfies the equation \eqref{3.4} and the conditions \eqref{4.18},
\eqref{4.20}. Conversely let $\l\in\bC_+$ and a function $y\in\cD$ satisfies
\eqref{3.4} and \eqref{4.18}. Then there exists $h=\{\hat h, h'\}\in
\hat\cK\oplus \cK'$ such that $y=Z_+(t,\l)(C_0(\l)-C_1(\l)M_+(\l))^{-1}h$ and
by \eqref{4.32} one has $(C_0(\l)\G_0-C_1(\l)\G_1)y=h$. Therefore in view of
\eqref{4.18} $h'=0$, so that $h=\hat h\in\hat\cK$ and by \eqref{4.29}
$y=v(t,\l) \hat h$. Similarly by using \eqref{4.33} one proves the same
statement in the case $\l\in\bC_-$.
\end{proof}
\begin{remark}\label{rem4.8a}
One can easily verify that for a given operator pair $N=(N_0\;\;N_1)$ the
operator function  $v(\cd,\l)$ is uniquely defined by the equivalence class
$\cP\in\TR$, i.e., $v(\cd,\l)$ does not depend on the choice of an
$N$-triangular Nevanlinna collection \eqref{4.2}-\eqref{4.5} inside the
equivalence class. To emphasize this fact we will write
$v(\cd,\l)=v_\cP(\cd,\l)$.
\end{remark}
\begin{theorem}\label{th4.9}
Assume that  $\Pi=\bta$ is a decomposing $D$-triplet \eqref{3.10} for $L$,
$\cP=\CD\in\TR$ is a collection  \eqref{4.2}-\eqref{4.5} and
$\f_N(\cd,\l):\D\to [\hat\cK,H], \;\l\in\bC$ is the operator solution of
\eqref{3.4} with the initial data
\begin{gather}\label{4.34}
\f_N^{(1)}(0,\l)=-N_0^*, \quad \f_N^{(2)}(0,\l)=N_1^*, \quad \l\in\bC.
\end{gather}
Then the generalized resolvent $\R=R_\cP(\l)$ generated by the boundary problem
\eqref{4.16}-\eqref{4.20} admits the representation \eqref{3.30} with the Green
function $G(x,t,\l)=G_\cP (x,t,\l)$ given by
\begin{gather}\label{4.36}
G_\cP (x,t,\l)=\begin{cases} v_\cP(x,\l)\f_N^*(t,\ov\l), \quad x>t \cr
\f_N(x,\l)v_\cP^*(t,\ov\l), \quad x<t \end{cases}, \quad \l\in\CR.
\end{gather}
\begin{proof}
Let $\pair\in\RH$ be a collection given by \eqref{2.14} and \eqref{3.25},
\eqref{3.26}, and let $Y_+(\cd,\l):\D\to [\cK_1,H], \;\l\in\bC_+,$ and
$Y_-(\cd,z):\D\to [\cK_0,H], \;z\in\bC_-,$ be the operator solutions of the
equation \eqref{3.4} with the initial data
\begin{gather}\label{4.37}
\wt Y_+(0,\l)=(-\hat D_0^*(\ov\l)\;\; \hat D_1^*(\ov\l))^\top, \;\;\;\; \wt
Y_-(0,z)=(-\hat C_0^*(\ov z)\;\; \hat C_1^*(\ov z ))^\top.
\end{gather}
Assume also that $\cZ_+(\cd,\l)\in \LH{\cK_0}$ and $\cZ_-(\cd,z)\in \LH{\cK_1}$
are  given by
\begin{gather}
\cZ_+(t,\l)=Z_+(t,\l)(C_0(\l)-C_1(\l)M_+(\l))^{-1},\;\;\l\in\bC_+\label{4.38}\\
\cZ_-(t,z)=Z_-(t,z)(\wt D_1(z)-\wt D_0(z)M_-(z))^{-1},
\;\;z\in\bC_-\label{4.38a}
\end{gather}
and let
\begin{gather*}
Y(t,\l)=\begin{cases} Y_+(t,\l),  \;\;\l\in\bC_+ \cr Y_-(t,\l),
\;\;\l\in\bC_-\end{cases}; \qquad \cZ (t,\l)=\begin{cases} \cZ_+(t,\l),
\;\;\l\in\bC_+ \cr \cZ_-(t,\l), \;\;\l\in\bC_-\end{cases}.
\end{gather*}
Then according to Theorem 14 in \cite{Mog10} the Green function in \eqref{3.30}
is
\begin{gather}\label{4.39}
G(x,t,\l)=\begin{cases} \cZ (x,\l) Y^*(t,\ov\l), \quad x>t \cr
Y(x,\l)\cZ^*(t,\ov\l), \quad x<t \end{cases}, \quad \l\in\CR.
\end{gather}
Next,  in the case of the block representations \eqref{4.2}-\eqref{4.5} one has
\begin{gather*}
\hat C_j(\l)=(N_j \;\; 0)^\top \in [H^n,\hat\cK\oplus \cK'_0], \quad
 \hat D_j(\l)=(N_j \;\; 0)^\top\in [H^n,\hat\cK\oplus \cK'_1], \;\;
j\in\{0,1\}.
\end{gather*}
Therefore the initial data \eqref{4.37} can be written in the form $\wt Y_+
(0,\l)=\begin{pmatrix} -N_0^* & 0 \cr N_1^* & 0\end{pmatrix}\in
[\hat\cK\oplus\cK_1',H^n\oplus H^n],\; \wt Y_- (0,z)=\begin{pmatrix} -N_0^* & 0
\cr N_1^* & 0\end{pmatrix}\in [\hat\cK\oplus\cK_0',H^n\oplus H^n]$, which in
view of \eqref{4.34} gives the block representations
\begin{gather}\label{4.40}
Y_+(t,\l)=(\f_N(t,\l)\;\;\;0):\hat\cK\oplus \cK_1'\to H,\;\;\;\;
Y_-(t,z)=(\f_N(t,z)\;\;\;0):\hat\cK\oplus \cK_0'\to H.
\end{gather}
Moreover by \eqref{4.29} the operator functions \eqref{4.38} have the block
 representations
\begin{gather}\label{4.41}
\cZ_+(t,\l)=(v_\cP(t,\l)\;\;\; u_+(t,\l)),\quad \cZ_-(t,z)=(v_\cP(t,z)\;\;\;
u_+(t,z))
\end{gather}
with some operator functions $u_+ (t,\l)$ and $u_-(t,z)$. Now combining
\eqref{4.40} and \eqref{4.41} with \eqref{4.39} we arrive at the equality
\eqref{4.36}.
\end{proof}
\end{theorem}
\subsection{$m$-functions}  Let $\Pi=\bta$ be a decomposing $D$-triplet \eqref{3.10} for $L$
 and let $N=(N_0\;\;N_1)$ be an
admissible operator pair \eqref{4.1}. Since $\cR (N)=\hat\cK$, it follows that
$\Ker N^*=\{0\}$ and $\cR (N^*)$ is a closed subspace in $H^n\oplus H^n$.
Therefore there exists a Hilbert space $\hat\cK^\perp$ and operators $T_j\in
[H^n,\hat\cK^\perp], \;j\in\{0,1\}$, such that the operator
\begin{gather}\label{4.42}
W'=\begin{pmatrix} -N_0^* & -T_0^* \cr N_1^* & T_1^* \end{pmatrix}:\hat
\cK\oplus\hat\cK^\perp\to \HH
\end{gather}
is an isomorphism.

Next assume that $W'$ is an isomorphism \eqref{4.42} and let
$Y_{W'}(\cd,\l)(\in [\hat \cK\oplus\hat\cK^\perp,H])$ be the operator solution
of the equation \eqref{3.4} such that $\wt Y_{W'} (0,\l)=W'$. Then
\begin{gather}\label{4.43}
Y_{W'}(t,\l)=(\f_N(t,\l) \;\;\;\f_T(t,\l)):\hat \cK\oplus\hat\cK^\perp\to H,
\quad\l\in\bC,
\end{gather}
where $\f_T(\cd,\l):\D\to [\hat\cK^\perp,H]$ is the operator solution of
\eqref{3.4} given by \eqref{4.34} with $T$ in place of $N$. Introduce also the
operator $ \cJ_{W'}=(W')^{-1}J_{H^n}(W')^{-1*}(\in [\hat
\cK\oplus\hat\cK^\perp]) $ where $J_{H^n}$ is the operator \eqref{3.30a}. Since
$\cJ_{W'}^*=-\cJ_{W'}$, the operator $\cJ_{W'}$ has the block  representation
\begin{gather}\label{4.47}
\cJ_{W'}=\begin{pmatrix}\cJ_1 & -\cJ_2^* \cr \cJ_2 & \cJ_4 \end{pmatrix}:\hat
\cK\oplus\hat \cK^\perp \to \hat\cK\oplus\hat \cK^\perp
\end{gather}
with $\cJ_1=-\cJ_1^*$ and $\cJ_4=-\cJ_4^*$.
\begin{theorem}\label{th4.11}
Assume that the following assumptions (a) are satisfied:

(a) $\;\Pi=\bta$ is a decomposing $D$-triplet \eqref{3.10} for $L$, $N=(N_0\;\;
N_1)$ is an operator pair \eqref{4.1},  $\cP=\CD\in\TR$ is a collection of
holomorphic pairs  \eqref{4.2}-\eqref{4.5}, $\pair\in\RH$ is the corresponding
collection \eqref{2.14} and $\Om_\tau(\cd)$ is the characteristic matrix
\eqref{3.24}.

Moreover, let $W'$ be an isomorphism \eqref{4.42} and let
$\Om_{\tau,W'}(\cd):\CR\to [\hat \cK\oplus\hat \cK^\perp]$ be the operator
function given by
\begin{gather}\label{4.48}
\Om_{\tau,W'}(\l)=(W')^{-1}\Om_\tau(\l)(W')^{-1*}, \quad \l\in\CR.
\end{gather}
Then: 1) The Green function \eqref{4.36} admits the representation
\begin{gather}\label{4.49}
G_\cP (x,t,\l)=Y_{W'}(x,\l)(\Om_{\tau,W'}(\l)+\tfrac 1 2 \text{sign}
(t-x)\cJ_{W'}) )Y_{W'}^*(t,\ov\l);
\end{gather}

2) The operator function \eqref{4.48} has the block  representation
\begin{gather}\label{4.50}
\Om_{\tau,W'}(\l)=\begin{pmatrix} m_\cP(\l) & -\tfrac 1 2 \cJ_2^*\cr -\tfrac 1
2 \cJ_2 & 0\end{pmatrix}:\hat\cK\oplus\hat \cK^\perp\to \hat\cK\oplus\hat
\cK^\perp, \quad \l\in\CR;
\end{gather}
3) The equality \eqref{4.50} generates the holomorphic operator function
$m_\cP(\cd):\CR\to [\hat\cK]$ which can be also defined by the following
statement:

(i) there exists a unique operator function $m_\cP(\cd):\CR\to [\hat\cK]$ such
that for every $\l\in\CR$ the operator function
\begin{gather}\label{4.51}
v(t,\l):=\f_N(t,\l)(m_\cP(\l)-\tfrac 1 2 \cJ_1)-\f_T(t,\l)\cJ_2
\end{gather}
belongs to $\LH{\hat\cK}$ and satisfies the boundary conditions
\eqref{4.25}-\eqref{4.28}.
\end{theorem}
\begin{proof}
1) The representation \eqref{4.49} is immediate from \eqref{3.31} and the
obvious equality $Y_0(t,\l)=Y_{W'}(t,\l)\,(W')^{-1},\;\l\in\bC$.

2) Let $v_\cP(\cd,\l)$ be the operator function defined in Proposition
\ref{pr4.8} and let
\begin{gather*}
u(x,\l)=(v_\cP(x,\l)\;\;0):\hat\cK\oplus\hat \cK^\perp\to H, \quad\l\in\CR.
\end{gather*}
Comparing \eqref{4.36} with \eqref{4.49} one obtains
\begin{gather}\label{4.52}
u(x,\l) Y_{W'}^*(t,\ov\l)=Y_{W'}(x,\l)(\Om_{\tau,W'}(\l)-\tfrac 1 2 \cJ_{W'})
Y_{W'}^*(t,\ov\l), \quad x>t
\end{gather}
for all $\l\in\CR$. Since $0\in\rho (\wt Y_{W'}(t,\ov\l))$, it follows from
\eqref{4.52} that
\begin {equation}\label{4.53}
u(x,\l)=Y_{W'}(x,\l)(\Om_{\tau,W'}(\l)-\tfrac 1 2 \cJ_{W'}), \quad
x\in\D,\;\;\l\in \CR.
\end{equation}
Next assume that the block  representation of the operator function
$\Om_{\tau,W'}(\l)$ is
\begin {equation}\label{4.54}
\Om_{\tau,W'}(\l)=\begin{pmatrix}m_\cP(\l) & \Om_3(\l) \cr \Om_2(\l) &
\Om_4(\l)\end{pmatrix}:\hat\cK\oplus\hat \cK^\perp\to \hat\cK\oplus\hat
\cK^\perp, \quad \l\in\CR.
\end{equation}
 Then the equality \eqref{4.53} can be written as
\begin {equation*}
(v_\cP (x,\l)\;\;0)=(\f_N(x,\l)\;\;\f_T(x,\l))\begin{pmatrix}m_\cP(\l)-\tfrac 1
2 \cJ_1 & \Om_3(\l)+\tfrac 1 2 \cJ_2^* \cr \Om_2(\l)-\tfrac 1 2 \cJ_2 &
\Om_4(\l)-\tfrac 1 2 \cJ_4 \end{pmatrix},
\end{equation*}
which implies the relations
\begin{gather}
v_\cP(x,\l)=\f_N(x,\l)(m_\cP(\l)-\tfrac 1 2 \cJ_1 )+\f_T(x,\l)(\Om_2(\l)-\tfrac
1 2
\cJ_2)\label{4.55}\\
\Om_3(\l)+\tfrac 1 2 \cJ_2^*=0,\quad \Om_4(\l)-\tfrac 1 2 \cJ_4=0, \quad
\l\in\CR.\label{4.56}
\end{gather}
Since $\Om_\tau(\l)=\Om_\tau^*(\ov\l)$, it follows from \eqref{4.48} that
$\Om_{\tau,W'}(\l)=\Om_{\tau,W'}^*(\ov\l)$ and by \eqref{4.54} one has
$\Om_2(\l)=\Om_3^*(\ov\l), \; \;\Om_4(\l)=\Om_4^*(\ov\l)$. Combining these
relations with \eqref{4.56} and taking the equality $\cJ_4=-\cJ_4^*$ into
account one obtains
\begin {equation}\label{4.57}
\Om_3(\l)=-\tfrac 1 2 \cJ_2^*, \quad \Om_2(\l)=-\tfrac 1 2 \cJ_2, \quad
\Om_4(\l)=0, \quad \l\in\CR.
\end{equation}
Therefore the block matrix representation \eqref{4.54} takes the
form\eqref{4.50}.

3) In view of \eqref{4.55} and the second equality in \eqref{4.57}  the
function $v(\cd,\l)=v_\cP(\cd,\l)$ admits the representation \eqref{4.51}. This
and Proposition \ref{pr4.8} give the statement 3).
\end{proof}
\begin{definition}\label{def4.12}
The operator function $\ml$ introduced in Theorem \ref{th4.11} will be called
an $m$-function corresponding to the collection $\cP\in\TR$ or, equivalently,
to the boundary value problem \eqref{4.16}-\eqref{4.20}.

The $m$-function $\ml$ will be called canonical if $\cP\in \TRZ$ or,
equivalently, if it corresponds to the canonical boundary problem \eqref{4.16},
\eqref{4.21}, \eqref{4.22}.
\end{definition}
\begin{remark}\label{rem4.13}
Let under the conditions of Theorem \ref{th4.11} $W'$ and $\wt W'$ be different
isomorphisms \eqref{4.42} (with the same first column), let
$\Om_{\tau,W'}(\cd)$ and $\Om_{\tau,\wt W'}(\cd)$ be the corresponding
functions \eqref{4.48} and let $m_\cP(\l)$ and $\wt m_\cP(\l)$ be upper left
entries in the representations \eqref{4.50}. One can easily verify that $\wt
m_\cP(\l)=m_\cP(\l)+ C, \;\; C=C^*$, which implies that the $m$-function $\ml$
is defined by a collection $\cP\in\TR$ up to the selfadjoint constant.
\end{remark}
For a given operator pair \eqref{4.1} introduce the operator $N'\in [H^n\oplus
H^n, \hat \cK]$ and the subspaces $\t$ and $\t^\perp$ in $\HH$ by
\begin {equation}\label{4.58}
N'=(-N_0\;\;\; N_1):\HH\to\hat\cK, \qquad \t^\perp=\Ker N', \quad
\t=(\HH)\ominus \t^\perp.
\end{equation}
Clearly, the operator $N_0':=N'\up\t$ isomorphically maps $\t$ onto $\hat\cK$
and the operator
\begin {equation}\label{4.58a}
\hat N:=(N_0')^{-1}, \quad \hat N\in [\hat\cK,\HH])
\end{equation}
 is the right inverse for $N'$, that is
$N'\hat N=I_{\hat\cK}$. Assume also that
\begin {equation}\label{4.59}
\hat N=(\hat N_0 \;\;\;\hat N_1)^\top:\hat\cK\to \HH
\end{equation}
is the block matrix representation of the operator $\hat N$.
\begin{proposition}\label{pr4.14}
Let the assumptions (a) of Theorem \ref{th4.11} be satisfied and let $\a
(\l)(\in [\cK_-,\hat\cK] ), \;\l\in\bC_-$ be a linear fractional transformation
of the Weyl function $M_-(\l)$ given by
\begin {equation}\label{4.60}
\a(\l)=(\hat N_1^*P_{H^n}-\hat N_0^*P_{H^n}M_-(\l))(\wt D_1(\l)-\wt
D_0(\l)M_-(\l))^{-1}, \quad\l\in\bC_-.
\end{equation}
Then: 1) the $m$-function $\ml$ is a uniformly strict Nevanlinna function
satisfying the relations
\begin{gather}
m_\cP(\mu)-m_\cP^*(\l)=(\mu-\ov\l)\smallint_0^b v_\cP^*(t,\l)v_\cP(t,\mu)\,
dt\,- \a(\ov\l)\bigl (D_{1}(\ov\l)D_{01}^*(\ov\mu)-\label{4.61}\\
-D_{01}(\ov\l)D_1^*(\ov\mu) +i D_{02}(\ov\l)D_{02}^*(\ov\mu)\bigr
)\a^*(\ov\mu),
\quad\; \mu,\l\in \bC_+\nonumber\\
(Im\,\l)^{-1}\cd Im\,(m_\cP(\l))\geq\smallint_0^b v_\cP^*(t,\l)v_\cP(t,\l)\,dt,
\quad \l\in\bC_+.\label{4.62}
\end{gather}
Here $D_{01}(\cd)$ and $D_{02}(\cd)$ are taken from \eqref{2.8b} and the
integral converges strongly, that is
\begin {equation*}
\smallint_0^b v_\cP^*(t,\l)v_\cP(t,\mu)\, dt=s-\lim_{\eta\uparrow b}
\smallint_0^\eta  v_\cP^*(t,\l)v_\cP(t,\mu)\, dt.
\end{equation*}
For the canonical $m$-function $\ml$  the identity \eqref{4.61} takes the form
\begin {equation}\label{4.63}
m_\cP(\mu)-m_\cP^*(\l)=(\mu-\ov\l)\smallint_0^b v_\cP^*(t,\l)v_\cP(t,\mu)\, dt,
\quad \mu, \l\in \CR
\end{equation}
and the inequality \eqref{4.62} turns into the equality.

 2) The characteristic matrix $\Om_\tau (\cd )$ admits the representation
\begin {equation}\label{4.64}
\Om_\tau (\l)= \begin{pmatrix}\Om_0(\l) & \Om_1^* \cr \Om_1 & \Om_2
\end{pmatrix}:\t\oplus\t^\perp\to \t\oplus\t^\perp, \quad \l\in \CR,
\end{equation}
where $\Om_2=\Om_2^*\in [\t^\perp]$ and $\Om_0(\cd):\CR\to [\t]$ is a uniformly
strict Nevanlinna function associated with $\ml$ by
\begin {equation}\label{4.65}
\Om_0(\l)=N_0'^*\, m_\cP(\l)N_0'+C, \quad C=C^*\in [\t].
\end{equation}
Moreover the following equality holds
\begin {equation}\label{4.66a}
m_\cP(\l)=\hat N^* \Om_\tau(\l) \hat N +\hat C, \qquad \hat C=\hat C^*\in
[\hat\cK], \quad \l\in\CR.
\end{equation}
\end{proposition}
\begin{proof}
 Let $W'$ be an isomorphism \eqref{4.42} and let $\Om_{\tau,W'}(\l)$ be the
operator function \eqref{4.48}. Then $\Om_\tau(\l)=W'\Om_{\tau,W'}(\l)W'^*$ and
the immediate calculation with taking \eqref{4.50} into account shows that
\begin {equation}\label{4.66}
\Om_\tau(\l)=N'^*\, m_\cP(\l)N'+\wt C
\end{equation}
with some $\wt C=\wt C^*\in [\HH]$. Multiplying the equality \eqref{4.66} by
$\hat N^*$ from the left and by $\hat N$ from the right one obtains
\eqref{4.66a}. Therefore $\ml$ is a Nevanlinna function.

In view of \eqref{4.66a} and \eqref{3.36}
\begin{gather}
m_\cP(\mu)-m_\cP^*(\l)=(\mu-\ov\l)\g_c^*(\l)\g_c(\mu)- \a(\ov\l)\bigl (D_{1}
(\ov\l)D_{01}^*(\ov\mu)-\label{4.67}\\
-D_{01}(\ov\l)D_1^*(\ov\mu) +i D_{02}(\ov\l)D_{02}^*(\ov\mu)\bigr
)\a^*(\ov\mu), \quad\; \mu,\l\in \bC_+,\nonumber
\end{gather}
where
\begin {equation}\label{4.68}
\g_c(\l)=\g_\tau(\l) \hat N, \quad \l\in\bC_+; \qquad \a(\l)=\hat N^*\wt
\a(\l), \quad \l\in \bC_-
\end{equation}
and $\g_\tau(\cd)$ and $\wt\a (\cd)$ are the operator functions \eqref{3.34}
and \eqref{3.35} respectively. Moreover by  \eqref{4.67} and the inequality in
\eqref{2.10} one has
\begin {equation}\label{4.68.1}
(Im\,\l)^{-1}\cd Im\,(m_\cP(\l))\geq \g_c^*(\l)\g_c(\l), \quad \l\in\bC_+,
\end{equation}
It follows from \eqref{4.2} and  \eqref{4.3} that for all $\l\in\bC_+$ the
operator $(-\hat C_0(\l): \hat C_1(\l))$ in \eqref{3.34} coincides with $N'$.
This and the equality $N'\hat N=I_{\hat\cK}$ imply that
\begin {equation}\label{4.68a}
\g_c(\l)=\g_+(\l)(C_0(\l)-C_1(\l)M_+(\l))^{-1}\up\hat\cK, \quad \l\in\bC_+.
\end{equation}
and, consequently, $0\in\rho (\g_c^*(\l)\g_c(\l))$. Therefore by \eqref{4.68.1}
the Nevanlinna function $\ml$ is uniformly strict. Moreover, in view of
\eqref{4.68a} and \eqref{4.29} one has $(\g_c(\l)\hat h)(t)=v_\cP(t,\l)\hat h
\;( \hat h\in\hat\cK)$. Applying now Lemma 4.1, 3) from \cite{Mog09.2} to
$v_\cP(t,\l)$ we arrive at the equality
\begin {equation*}
\g_c^*(\l)f=\smallint_0^b v_\cP^*(t,\l)f(t)\, dt:=\lim_{\eta\uparrow b}
\smallint_0^\eta  v_\cP^*(t,\l)f(t)\, dt, \quad f=f(t)\in\gH,
\end{equation*}
which implies that
\begin {equation}\label{4.70}
\g_c^*(\l)\g_c(\mu)=\smallint_0^b v_\cP^*(t,\l)v_\cP(t,\mu)\,
dt:=s-\lim_{\eta\uparrow b} \smallint_0^\eta  v_\cP^*(t,\l)v_\cP(t,\mu)\, dt.
\end{equation}

Next, in view of \eqref{3.35} and \eqref{4.59} the second equality in
\eqref{4.68} can be written as
\begin {gather*}
\a(\l)=(\hat N_0^* :\hat N_1^*)\begin{pmatrix} -P_{H^n}M_-(\l)\cr P_{H^n}
\end{pmatrix} (\wt D_1(\l)-\wt D_0(\l)M_-(\l))^{-1}=\\
(\hat N_1^*P_{H^n}-\hat N_0^*P_{H^n}M_-(\l))(\wt D_1(\l)-\wt
D_0(\l)M_-(\l))^{-1}.
\end{gather*}
Therefore the operator function $\a(\l)$ defined by \eqref{4.68} can be
represented in the form \eqref{4.60}. Combining this assertion with
\eqref{4.67}, \eqref{4.68.1} and \eqref{4.70} we obtain the identity
\eqref{4.61} and the inequality \eqref{4.62}. Moreover \eqref{4.61} and the
equality \eqref{2.12a} yield \eqref{4.63}.

Finally, the equality \eqref{4.64}  is immediate from \eqref{4.66} and the
block representation $ N'=(N_0'\;\;0):\t\oplus\t^\perp\to\hat\cK. $
\end{proof}
\begin{corollary}\label{cor4.14a}
Let the assumptions  (a) of Theorem \ref{th4.11} be satisfied. Then the
following statements are equivalent:

(i) the characteristic matrix $\Om_\tau(\cd)$ is a uniformly strict Nevanlinna
function;

(ii) the operator $N=(N_0\;\;N_1)$  \eqref{4.1}  isomorphically maps $\HH$ onto
$\hat\cK$.

If in addition $\dim H<\infty$, then the  statement (i) is equivalent to the
following one:

(iii) the operator $L_0$ has maximal  deficiency indices
$n_+(L_0)=n_-(L_0)=2n\dim H$, $\cH_0'=\cH_1'=:\cH'$ (i.e., $\Pi=\{H^n\oplus
\cH', \G_0,\G_1 \}$ is a decomposing boundary triplet for $L$), $\dim
\cH'=n\dim H$ and the collection $\cP$ can be represented as the  holomorphic
Nevanlinna pair (c.f. Remark \ref{rem2.3}, 2)) $C(\l)=(C_0(\l)\;\; C_1(\l)),
\;\;\l\in\CR$,
\begin {equation}\label{4.71}
C_0(\l)=(N_0\;\;C_0'(\l)):H^n\oplus \cH'\to\hat\cK, \quad
C_1(\l)=(N_1\;\;C_1'(\l)):H^n\oplus \cH'\to\hat\cK,
\end{equation}
where $\dim\hat\cK=2n\dim H$ and the operator $N=(N_0\;\,N_1):\HH\to\hat\cK $
is an isomorphism.
\end{corollary}
\begin{proof}
It follows from \eqref{4.64} that the Nevanlinna function $\Om_\tau(\cd)$ is
uniformly strict if and only if $\t^\perp=\{0\}$. Moreover by \eqref{4.58} one
has $\t^\perp=\{0\}\Leftrightarrow \Ker N(=\Ker N')=\{0\}$. This yields the
equivalence $(i)\Leftrightarrow (ii)$.

Next assume that $\dim H<\infty$ and prove the equivalence $(ii)\Leftrightarrow
(iii)$. If $0\in\rho (N)$, then $\dim\hat\cK=\dim (\HH)=2n\,\dim H$ and by
\eqref{4.10} $n_-(L_0)=n_+(L_0)=2n\,\dim H$. This and the second relation in
\eqref{3.13} imply that $\dim \cH_1=\dim\cH_0=2n\,\dim H$  and hence
$\cH_0=\cH_1=:\cH$. Therefore $\cH_0'=\cH_1'$ and $\Pi=\{H^n\oplus \cH',
\G_0,\G_1 \}$ is a decomposing boundary triplet for $L$. Moreover by
\eqref{2.14b} the Hilbert spaces $\hat\cK\oplus \cK_j'$ in
\eqref{4.2}-\eqref{4.5} satisfy the equalities $\dim (\hat\cK\oplus
\cK_j')=\dim \cH=2n\, \dim H=\dim \hat\cK$. Hence $\cK_j'=\{0\},\;j\in \{0,1\}$
and the equalities \eqref{4.2}-\eqref{4.5} take the form \eqref{4.71}, which
yields the implication $(ii)\Rightarrow (iii)$. The inverse implication
$(iii)\Rightarrow (ii)$ is obvious. Thus in the case $\dim H<\infty$  the
equivalences $(i)\Leftrightarrow (ii)\Leftrightarrow (iii)$ hold.
\end{proof}
\begin{remark}\label{rem4.14b}
1) It follows from Corollary \ref{cor4.14a} that in the case $\dim H<\infty$
and $n_-(L_0)< 2n\,\dim H$ the characteristic matrix $\Om_\tau (\cd)$
corresponding to the boundary operators \eqref{4.2}-\eqref{4.5} is not
uniformly strict Nevanlinna function. In particular, by Proposition
\ref{pr4.4a} this statement holds for each canonical characteristic matrix
corresponding to the constant Nevanlinna collection \eqref{2.12}.

2) Let $\cP_0\in\TR$ be the collection \eqref{4.2}-\eqref{4.5} with
$\hat\cK=H^n,\; \cK_j'=\cH_j',\; j\in \{0,1\}$ and $C_0(\l) =I_{\cH_0},\;
C_1(\l)=0_{\cH_1,\cH_0},\;D_0=P_1$ and $D_1=0_{\cH_1}$. Then the corresponding
$m$-function $m_{\cP_0}(\cd)$ coincides with the operator function $m(\cd)$
defined by \eqref{3.21} and \eqref{3.22}. Note in this connection that the
statements of Theorem \ref{th4.11} and Proposition \ref{pr4.14} for
$m(\l)(=m_{\cP_0}(\l))$ were obtained in our paper \cite{Mog09.1}.
\end{remark}
\subsection{m-function and a characteristic matrix as the Weyl functions}
In this subsection we show that a canonical $m$-function $\ml$ is the Weyl
function  of some symmetric extension $\hat A\in \exl$, while a canonical
characteristic matrix $\Om_\tau (\cd)$ is the Weyl function of the minimal
operator $L_0$ (the last statement holds under some additional assumptions).
\begin{proposition}\label{pr4.15}
Assume that $\Pi=\bt$ is a decomposing  boundary triplet \eqref{3.10} for $L$
(with $\cH=H^n\oplus\cH'$) and $\cP=\{C_0,C_1\}\in TR^0 (\cH)$ is an operator
pair defined by \eqref{4.7} with $\cH_0'=\cH_1'=:\cH'$. Moreover let $X_j\in
[\cH,\cK], \; j\in\{0,1\}$ be operators such that the operator
\begin {equation}\label{4.73}
X=\begin{pmatrix} C_0 & -C_1 \cr X_0 & -X_1
\end{pmatrix}:\cH\oplus\cH\to\cK\oplus\cK
\end{equation}
satisfies the relations $X^* J_\cK X=J_\cH$ and $0\in\rho (X)$ (such operators
$X_0$ and $X_1$ exist, because $\{(C_0, -C_1);\cK\}$ is a selfadjoint operator
pair). Suppose also that
\begin {equation}\label{4.74}
X_0=\begin{pmatrix} X_{01} & X_{02} \cr X_{03} & X_{04} \end{pmatrix}: H^n
\oplus\cH'\to \hat\cK\oplus\cK', \quad X_1= \begin{pmatrix}
 X_{11} & X_{12} \cr X_{13} & X_{14} \end{pmatrix}: H^n
\oplus\cH'\to \hat\cK\oplus\cK'
\end{equation}
are the block matrix representations of $X_j, \;j\in\{0,1\}$ and
\begin {equation}\label{4.75}
\wt C_0':=\begin{pmatrix} C_{01}' \cr C_{02}' \cr X_{02}\end{pmatrix}:\cH'\to
\hat\cK \oplus \cK'\oplus \hat\cK, \qquad \wt C_1':=\begin{pmatrix} C_{11}' \cr
C_{12}' \cr X_{12}\end{pmatrix}:\cH'\to \hat\cK \oplus \cK'\oplus \hat\cK.
\end{equation}

Then: 1) the operator $\hat A:=L\up \cD (\hat A)$ defined by the decomposing
boundary conditions
\begin {equation}\label{4.76}
\cD (\hat A)=\{y\in\cD: \;y^{(1)}(0)=y^{(2)}(0)=0, \; \wt C_0'\G_0'y- \wt
C_1'\G_1'y=0\}
\end{equation}
is a closed symmetric extension of $L_0$;

2) the adjoint $\hat A^*$ of $\hat A$ is defined by the boundary condition
\begin {equation}\label{4.77}
\cD (\hat A^*)=\{y\in\cD :\, C_{02}'\G_0'y-C_{12}'\G_1'y=0\};
\end{equation}

3) the maps $\hat\G_j:\cD (\hat A^*)\to \hat \cK, \; j\in\{0,1\}$ given by
\begin {gather}
\hat \G_0y= N_0 y^{(2)}(0)+N_1 y^{(1)}(0)+C_{01}'\G_0'y-C_{11}'
\G_1'y,\qquad\label{4.78}\\
\qquad\hat \G_1y= X_{01} y^{(2)}(0)+X_{11} y^{(1)}(0)+X_{02}\G_0'y-X_{12}
\G_1'y, \quad  y\in \cD (\hat A^*)\label{4.79}
\end{gather}
form a boundary triplet $\hat\Pi=\{\hat \cK, \hat \G_0,\hat\G_1\}$ for $\hat
A^*$;

4) the corresponding $\g$-field and  Weyl function \eqref{2.28} for $\hat\Pi$
are
\begin {gather}
(\hat \g(\l)h)(t)=v_\cP (t,\l)h, \quad h\in \hat\cK, \quad \l\in\CR
\label{4.80}\\
\hat M(\l)=m_\cP(\l)+\hat D, \quad \hat D=\hat D^*\in [\hat\cK], \quad
\l\in\CR.\label{4.81}
\end{gather}
\end{proposition}
\begin{proof}
According to \cite{DM00}  the operators
\begin {equation}\label{4.82}
\wt\G_0=C_0\G_0-C_1\G_1, \qquad \wt \G_1=X_0\G_0-X_1\G_1
\end{equation}
form a boundary triplet $\wt\Pi=\{\hat\cK \oplus \cK', \wt\G_0, \wt\G_1\}$ for
$L$ and the $\g$-field for $\wt\Pi$ is
\begin {equation}\label{4.83}
\wt \g(\l)=\g(\l) (C_0-C_1 M(\l))^{-1}, \quad \l\in\CR,
\end{equation}
where $\g(\cd)$ and $M(\cd)$ are the $\g$-field and the Weyl function for $\Pi$
respectively. Moreover by \eqref{4.7}, \eqref{4.74} and \eqref{3.10} the
equalities \eqref{4.82} can be written as
\begin{gather}
\wt\G_0 y=\{N_0 y^{(2)}(0)+N_1y^{(1)}(0) +C_{01}'\G_0'y-C_{11}'\G_1'y,\;\;
C_{02}'\G_0'y- C_{12}'\G_1'y\}\qquad\label{4.84}\\
\wt\G_1 y=\{X_{01} y^{(2)}(0)+X_{11}y^{(1)}(0)
+X_{02}\G_0'y-X_{12}\G_1'y,\qquad\qquad\qquad\qquad\label{4.85}\\
\qquad\qquad\qquad\qquad\qquad\qquad X_{03} y^{(2)}(0)+X_{13}y^{(1)}(0)+
X_{04}\G_0'y- X_{14}\G_1'y\}, \quad y\in\cD\nonumber
\end{gather}
Applying Proposition 4.1 from \cite{DM00} to the triplet $\wt \Pi$ one obtains
the following assertions:

(i) the equality $\cD (\hat A)=\{y\in \cD: \wt\G_0y=P_{\hat\cK}\wt\G_1 y=0\}$
defines a symmetric extension $\hat A\in\exl;$

(ii) the adjoint $\hat A^*$ of $\hat A$ is given by $\cD (\hat A^*)=\{y\in\cD:
P_{\cK'}\wt\G_0 y=0 \}$;

(iii) the operators $\hat\G_j:\cD (\hat A^*)\to \hat \cK, \; j\in\{0,1\}$ given
by
\begin {equation}\label{4.86}
\hat\G_0y=\wt\G_0 y, \quad \hat\G_1 y=P_{\hat\cK}\wt\G_1 y, \quad y\in\cD (\hat
A^* )
\end{equation}
form a boundary triplet $\hat\Pi=\{\hat \cK, \hat \G_0,\hat\G_1\}$ for $\hat
A^*$ and the $\g$-field for $\hat\Pi$ is
\begin {equation}\label{4.87}
\hat \g(\l)=\wt\g (\l)\up\hat\cK=\g(\l)(C_0-C_1 M(\l))^{-1}\up\hat\cK.
\end{equation}
Next we show that the operator $\hat A$ and the triplet $\hat\Pi$ satisfy the
statement 1)-4) of the proposition.

First observe that   \eqref{4.77} is immediate from the assertion (ii) and the
equality \eqref{4.84}. Next, \eqref{4.84}, \eqref{4.85} and the assertion (i)
yield the equivalence
\begin {equation}\label{4.88}
y\in\cD (\hat A)\iff \begin{cases} N_0 y^{(2)}(0)+N_1y^{(1)}(0)+C_{01}'\G_0'y-
C_{11}'\G_1'y=0 \cr C_{02}'\G_0'y- C_{12}'\G_1'y=0 \cr X_{01}
y^{(2)}(0)+X_{11}y^{(1)}(0) +X_{02}\G_0'y-X_{12}\G_1'y=0
 \end{cases}.
\end{equation}
Let $\cD_2'$ be the set of all functions $y\in\cD$ finite at the point $b$.
According to \cite{Mog09.1} $\G_0'\up\cD_2'=\G_1'\up\cD_2'=0$ and in view of
\eqref{4.77} $\cD_2'\subset \cD (\hat A^*)$. This and the Lagrange's  identity
\eqref{3.5} imply that $y^{(1)}(0)=y^{(2)}(0)=0$ for all $y\in\cD (\hat A)$.
Combining these equalities with \eqref{4.88} and taking \eqref{4.75} into
account we arrive at \eqref{4.76}.

The statement 3) of the proposition follows from \eqref{4.86} and the
equalities  \eqref{4.84}, \eqref{4.85}. Moreover combining \eqref{4.87} with
\eqref{3.18a} and \eqref{4.29} one obtains the equality \eqref{4.80} for $\hat
\g (\l) $. Finally to prove \eqref{4.81} note that the Weyl function $\hat
M(\cd)$ for $\hat \Pi$ satisfies the identity \cite{DM91}
\begin {equation}\label{4.89}
\hat M(\mu)-\Hat M^*(\l)=(\mu-\ov\l)\hat\g^*(\l)\hat \g(\mu), \quad
\mu,\l\in\CR.
\end{equation}
Moreover in view of \eqref{4.87} $\hat \g(\l)$ coincides with the operator
function $\g_c(\l)$ defined by \eqref{4.68a}. Therefore the equality
\eqref{4.70} holds with $\hat\g(\l)$ in place of $\g_c(\l)$ and by \eqref{4.63}
\begin {equation}\label{4.89a}
 m_\cP(\mu)-m_\cP^*(\l)=(\mu-\ov\l)\hat\g^*(\l)\hat \g(\mu),\quad \mu,\l\in\CR.
\end{equation}
Now comparing \eqref{4.89} and \eqref{4.89a} we arrive at the relation
\eqref{4.81}.
\end{proof}
\begin{corollary}\label{cor4.17}
Assume that $\Pi=\{H^n\oplus \cH', \G_0,\G_1\}$ is a decomposing boundary
triplet for $L$, $N=(N_0\;\;N_1)$ is an operator pair \eqref{4.1} such that
$0\in\rho (N)$, $N'$ is the operator \eqref{4.58}, $\cP=\{C_0,C_1\}\in
TR^0(\cH)$ is the operator pair given by (c. f. \eqref{4.71})
\begin {equation}\label{4.90}
C_0=(N_0\;\; C_0'): H^n\oplus \cH'\to \hat\cK, \qquad C_1=(N_1\;\; C_1'):
H^n\oplus \cH'\to \hat\cK
\end{equation}
and
\begin {equation*}
X=\begin{pmatrix} (N_0 \;\;\; C_0') &  (-N_1 \;\; -C_1') \cr  (X_{01} \;\;
X_{02}) &  (-X_{11} \; -X_{12}) \end{pmatrix}:(H^n\oplus \cH')\oplus (H^n\oplus
\cH')\to \hat\cK\oplus\hat\cK
\end{equation*}
is the operator satisfying the relations $X^* J_{\hat\cK}X=J_\cH$ and $0\in\rho
(X)$. Moreover let $\Om_\tau (\cd)$ be the canonical characteristic matrix
corresponding to $\tau =\{(C_0,C_1); \hat\cK\}$. Then the operators
\begin{gather}
\G_0^\Om y=\{-y^{(2)}(0), y^{(1)}(0)\}+N'^{-1}C_0'\G_0' y- N'^{-1}C_1'\G_1'y
\;(\in\HH)\qquad\qquad\label{4.91}\\
\G_1^\Om y=N'^*
(X_{01}y^{(2)}(0)+X_{11}y^{(1)}(0)+X_{02}\G_0'y-X_{12}\G_1'y)\;(\in\HH), \quad
y\in\cD\label{4.92}
\end{gather}
form a boundary triplet $\{\HH,\G_0^\Om,  \G_1^\Om \}$ for $L$ with the
corresponding Weyl function
\begin {equation}\label{4.92a}
M_\Om(\l)=\Om_\tau (\l) +\wt D, \qquad \wt D=\wt D^*\in [\HH].
\end{equation}
\end{corollary}
\begin{proof}
Comparing \eqref{4.7} with  \eqref{4.90} one obtains $\cK'=\{0\}$ and hence
$C_{02}'=C_{12}'=0$.  Therefore the extensions \eqref{4.77} and \eqref{4.76}
take the form $\hat A^*=L, \; \hat A=L_0$ and by Proposition \ref{pr4.15} the
operators \eqref{4.78} and \eqref{4.79} form the boundary triplet $\hat \Pi=
\{\hat \cK, \hat \G_0,\hat\G_1\} $ for $L$ with the Weyl function \eqref{4.81}.
Next, according to Proposition 3.13 in \cite{DM00} the collection
$\{\HH,\G_0^\Om,  \G_1^\Om \}$ with $\G_0^\Om=N'^{-1}\hat\G_0$ and $\G_1^\Om=
N'^*\hat\G_1$ is a boundary triplet for $L$ and the corresponding Weyl function
is
\begin {equation}\label{4.93}
M_\Om(\l)=N'^* \hat M(\l) N'
\end{equation}
Now the equalities \eqref{4.91} and \eqref{4.92} are implied by  \eqref{4.78},
\eqref{4.79} and the  relation
\begin {equation*}
N'^{-1}(N_0 y^{(2)}(0)+N_1 y^{(1)}(0))=N'^{-1}N'\{-y^{(2)}(0),
y^{(1)}(0)\}=\{-y^{(2)}(0), y^{(1)}(0)\}.
\end{equation*}
Moreover since $0\in\rho (N')(\Leftrightarrow 0\in\rho (N))$, it follows from
\eqref{4.58} that $\t^\perp =\{0\}$ and by \eqref{4.64}, \eqref{4.65} one has
\begin {equation*}
\Om_\tau(\l)=\Om_0(\l)=N'^* m_\cP(\l)N'+C, \qquad C=C^*\in [\HH].
\end{equation*}
Combining this equality with \eqref{4.81} and \eqref{4.93} one obtains the
relation \eqref{4.92a} for $M_\Om (\l)$.
\end{proof}
\begin{remark}\label{rem4.18}
 It follows from Corollary \ref{cor4.14a} that in the case $\dim H<\infty$ the
conditions of Corollary \ref{cor4.17} are necessary (and, by the statement of
this corollary, sufficient) for the canonical characteristic matrix
$\Om_\tau(\cd)$ be the Weyl function of the operator $L_0$.
\end{remark}
\section{Spectral functions of differential operators}
\subsection{The space $L_2(\S ; \cH)$}
Let $\cH$ be a separable Hilbert space.
\begin{definition}\label{def5.1}
A nondecreasing operator function $\S:\bR\to [\cH]$ is called a distribution if
it is strongly left continuous and satisfies the equality $\S(0)=0$.
\end{definition}
Let $\S:\bR\to [\cH]$ be a distribution and let $f(\cd), \; g(\cd)$ be vector
functions defined on the segment $[\a, \b]$ with values  in $\cH$. Consider the
Riemann-Stieltjes integral \cite{Ber}
\begin {equation}\label{5.1}
\int_\a^\b (d\,\S(t)f(t),g(t))=\lim_{d_\pi\to 0} \sum_{k=1}^n
((\S(t_k)-\S(t_{k-1}))f(\xi_k), g(\xi_k)),
\end{equation}
where $\pi=\{\a=t_0<t_1< \dots <t_n=\b\}$ is a partition of $[\a,\b]$,
$\;\xi_k\in [t_{k-1}, t_k]$ and $d_\pi$ is the diameter of  $\pi$. As is known
(see for instance \cite{MalMal03}) in the case $\dim \cH=\infty$ there exist a
distribution $\S(\cd)$ and continuous functions $f(\cd)$ and $g(\cd)$ for which
the integral \eqref{5.1} does not exist. At the same time holomorphy of
$f(\cd)$ and $g(\cd)$ on the segment $[\a, \b]$ is a sufficient condition for
existence of such an integral \cite{Shm71}.
\begin{definition}\label{def5.2}
A function $f:[\a,\b)\to \cH$ will be called piecewise holomorphic if there is
a partition $\a=t_0<t_1< \dots  <t_n=\b$ such that each restriction $f\up
[t_{k-1}, t_k)$ admits a holomorphic continuation $\wt f_k(\cd)$ on some
interval $(\wt t_{k-1}, \wt t_k)\supset [t_{k-1}, t_k]$.

A function $f:\bR\to \cH$ will be called piecewise holomorphic if it is so on
each finite half-interval $[\a, \b)$.
\end{definition}
It follows from Definition \ref{def5.2} that a piecewise holomorphic function
is strongly right continuous.

Let $\S:\bR\to [\cH]$ be a distribution and let $f,\; g: [\a, \b)\to \cH$ be a
pair of piecewise holomorphic functions. It is clear that there exists a
partition of $[\a,\b]$ satisfying the conditions of Definition \ref{def5.2} for
both functions $f(\cd)$ and $g(\cd)$. By using such a partition introduce the
integral
\begin {equation}\label{5.2}
\int_{[\a,\b)} (d\,\S(t)f(t),g(t))= \sum_{k=1}^n \int_{t_{k-1}}^{t_k} (d\,\S(t)
\wt  f_k(t), \wt g_k(t)).
\end{equation}
Note that for a pair of continuous functions $f,\; g: [\a, \b]\to \cH$
piecewise holomorphic on $[\a,\b)$ there exists the integral \eqref{5.1} which
coincides with that of \eqref{5.2}.

For a given distribution $\S:\bR\to[\cH]$ denote by $Hol(\S,\cH)$ the set of
all piecewise holomorphic functions $f:\bR\to \cH$ such that
\begin {equation*}
\int_{\bR}(d\,\S(t)f(t),f(t)):=\lim_{[\a,\b)\to \bR}\int _{[\a,\b)}
(d\,\S(t)f(t),f(t))<\infty.
\end{equation*}
One can easily prove that for each pair $f,g\in Hol(\S,\cH)$ there exists the
integral
\begin {equation}\label{5.3}
(f,g)_{Hol(\S,\cH)}=\int_{\bR}(d\,\S(t)f(t),g(t)):=\lim_{[\a,\b)\to \bR}\int
_{[\a,\b)} (d\,\S(t)f(t),g(t)).
\end{equation}
This implies that $Hol(\S,\cH)$ is a linear space with the semi-definite scalar
product \eqref{5.3}.

Next recall the definition of the space $L_2(\S;H)$ as it was given in the book
\cite{Ber}.

A function $f:\bR\to \cH$  is called finite-dimensional if there is a subspace
$\cH_f\subset \cH$ such that $\dim \cH_f<\infty$ and $f(t)\in \cH_f, \;
t\in\bR$. For a given distribution $\S:\bR\to[\cH]$ denote by $C_{00}(\cH)$ the
linear space of all strongly continuous finite-dimensional functions
$f:\bR\to\cH$ with compact support $supp \;f$. Clearly the equality
\begin {equation}\label{5.4}
(f,g)_{L_2(\S;\cH)}=\int_\bR (d\,\S(t)f(t),g(t)):=\int_\a^\b
(d\,\S(t)f(t),g(t)) , \quad f,g\in C_{00}(\cH)
\end{equation}
with $[\a,\b]\supset (supp \, f \cup supp\,g)$ defines the semi-definite scalar
product on $C_{00}(\cH)$. The completion of $C_{00}(\cH)$ with respect to this
 product is a semi-Hilbert space $\wt L_2(\S;\cH)$. The quotient of
$\wt L_2(\S;\cH)$ over the kernel $\{f\in \wt L_2(\S;\cH): (f,f)_{L_2
(\S;\cH)}=0\}$ is the Hilbert space $L_2(\S;H)$.

Denote by $Hol_0(\S,\cH)$ the set of all strongly continuous, piecewise
holomorphic and finite dimensional functions $f:\bR\to \cH$ with a compact
support. It is clear that $Hol_0(\S,\cH)=Hol(\S,\cH)\cap C_{00}(\cH)$ and
consequently $Hol_0(\S,\cH)$ is a linear manifold both in $Hol(\S,\cH)$ and
$C_{00}(\cH)$. Moreover the semiscalar products \eqref{5.3} and \eqref{5.4}
coincide on $Hol_0(\S,\cH)$.

By using the Taylor expansions of the function $f\in Hol(\S,\cH)$ one can prove
the following proposition.
\begin{proposition}\label{pr5.3}
The set $Hol_0(\S,\cH)$ is a dense linear manifold both in $Hol(\S,\cH)$ and
$C_{00}(\cH)$, which implies that the closure of  $ Hol(\S,\cH)$ coincides with
$ \wt L_2(\S;\cH)$. In other words the semi-Hilbert space $\wt L_2(\S;\cH)$ can
be considered as the completion of $Hol(\S,\cH)$.
\end{proposition}
\begin{remark}
In connection with Proposition \ref{pr5.3} note that the intrinsic functional
description of the spaces $\wt L_2(\S;\cH)$ and $L_2(\S;\cH)$ in the case
$\dim\cH<\infty$ was obtained in \cite{Kac50}. Moreover in the case
$\dim\cH=\infty$ the description of these spaces in terms of the direct
integrals of Hilbert spaces can be found in the recent paper \cite{MalMal03}.
\end{remark}
\subsection{Spectral functions}\label{sub4.2}
Let $\Pi=\bta$ be a decomposing $D$-triplet \eqref{3.10} for $L$ and let
$\pair\in\RH$ be a Nevanlinna collection defined by \eqref{2.14} and
\eqref{3.25}, \eqref{3.26}. For this collection consider the boundary problem
\eqref{3.27}-\eqref{3.29}. According to Remark \ref{rem3.5a} this problem
defines the spectral function $F_\tau(t)$ of the operator $L_0$.

Next assume that $\wt\cK$ is a separable Hilbert space and $\f(\cd,\l):\D\to
[\wt\cK,H]$ is an operator solution of the equation \eqref{3.4} with the
constant initial data $\wt\f(0,\l)\equiv\wt\f_0(\in [\wt\cK,
H^{2n}]),\;\l\in\bC,$ such that $0\in\hat\rho (\wt\f_0)$. Denote by $\gH_0$ the
set of all functions $f\in\gH(=L_2(\D;H))$ with $supp f\subset [0,\b]$ ($\b<b$
depends on $f$) and consider the Fourier transform $g_f:\bR\to\wt\cK$ of a
function $f\in\gH_0$ given by
\begin {equation}\label{5.6}
g_f(s)=\int_0^b \f^*(t,s) f(t)\,dt.
\end{equation}
\begin{definition}\label{def5.5}
A distribution $\S(\cd)=\S_{\tau,\f}(\cd):\bR\to [\wt\cK]$ is called a spectral
function of the boundary problem \eqref{3.27}-\eqref{3.29} corresponding to the
solution $\f(\cd,\l)$ if for each function $f\in\gH_0$ the Fourier transform
\eqref{5.6} satisfies the equality
\begin {equation}\label{5.7}
((F_\tau(\b)-F_\tau(\a))f,f)_\gH=\int_{[\a,\b)} (d\S_{\tau,\f}(s)g_f(s),
g_f(s)), \quad [\a,\b)\subset\bR.
\end{equation}
\end{definition}
Note that the integral in the right hand part of \eqref{5.7} exists, because
the function $g_f(\cd)$ is holomorphic on $\bR$. Moreover by \eqref{5.7}
$g_f(\cd)\in Hol(\S_{\tau,\f},\wt\cK)$ and the following Parseval equality
holds
\begin {equation*}
(||f||_\gH^2=)\int_0^b ||f(t)||_H^2 \,dt=\int_\bR (d\S_{\tau,\f}(s)g_f(s),
g_f(s)) \bigl (=||g_f||^2_{L_2(\S_{\tau,\f};\wt\cK)}\bigr), \quad f\in\gH_0.
\end{equation*}
This implies that  the linear operator $V:\gH\to L_2(\S_{\tau,\f};\wt\cK) $
defined on the dense linear manifold $\gH_0\subset\gH$ by $(Vf)(s)=g_f(s)$ is
an isometry.
\begin{definition}\label{def5.6}
A spectral function $\S_{\tau,\f}(\cd)$ is called orthogonal if $V\gH=
L_2(\S_{\tau,\f};\wt\cK)$ or equivalently if the set of all Fourier transforms
$\{g_f(\cd):f\in\gH_0\}$ is dense in $L_2(\S_{\tau,\f};\wt\cK)$.
\end{definition}
\begin{theorem}\label{th5.7}
Let $\wt H$ be a Hilbert space with $\dim \wt H=2n\cd\dim H$, let $W\in [\wt
H,H^{2n}]$ be an isomorphism and let $Y_W(\cd,\l):\D\to [\wt H,H] $ be an
operator solution of the equation \eqref{3.4} with the initial data $\wt
Y_W(0,\l)=W,\;\l\in\bC$.  Then for each collection $\tau\in\RH$ there exists
the unique spectral function $\S_{\tau, W}:\bR\to [\wt H]$ of the boundary
problem \eqref{3.27}-\eqref{3.29} corresponding to the solution $Y_W(\cd,\l)$.
This function is defined by the equality
\begin {equation}\label{5.9}
\S_{\tau, W}(s)=s-\lim_{\d\to+0} w-\lim_{\e\to +0}\frac 1 \pi\int_{-\d}^{s-\d}
Im \,\Om_{\tau, W}(\s+ i\,\e)\,d\s,
\end{equation}
where $\Om_{\tau, W}:\CR\to [\wt H]$ is a Nevanlinna operator function given by
\begin {equation}\label{5.9a}
\Om_{\tau, W}(\l)=W^{-1}\Om_\tau(\l)W^{-1*}, \quad \l\in\CR.
\end{equation}
Moreover the spectral function $\S_{\tau, W}(\cd)$ is orthogonal if and only if
$\tau\in\RZ$.
\end{theorem}
One can prove Theorem \ref{th5.7} by using the Stieltjes -Liv$\breve{\text
s}$ic formula \cite{KacKre,Shm71} in the same way as in \cite{Sht57} (the
scalar case $\dim H=1$) and \cite{Bru74} (the case $\dim H\leq \infty$).
Moreover in the scalar case other methods of the proof can be found in the
books \cite{Nai,DunSch}.
\begin{theorem}\label{th5.8}
Assume that under the conditions of  Theorem \ref{th5.7} $\S_{\tau, W}(\cd)$ is
a spectral function of the boundary problem \eqref{3.27}-\eqref{3.29} and
$V:\gH\to L_2(\S_{\tau,W};\wt H)$ is the corresponding isometry given by the
Fourier transform \eqref{5.6} with $\f (t,s)=Y_W(t,s)$. Moreover, let
$Hol^0(\wt H)$ be a linear manifold of all piecewise holomorphic functions
$g:\bR\to \wt H$ with compact support. Then $Hol^0(\wt H)$ is dense in
$L_2(\S_{\tau,W};\wt H)$ and
\begin {equation}\label{5.11}
(V^*g)(t)=\int_\bR Y_W(t,s)\, d\S_{\tau,W}(s)g(s), \quad g=g(s)\in Hol^0(\wt
H),
\end{equation}
where $V^*:L_2(\S_{\tau,W};\wt H)\to \gH$ is the adjoint operator and the
integral is understood as the sum of integrals similarly to \eqref{5.2}. In
particular formula \eqref{5.11} implies that the inverse Fourier transform is
\begin {equation}\label{5.12}
f(t)=\int_\bR Y_W(t,s)\, d\S_{\tau,W}(s)g_f(s).
\end{equation}
\end{theorem}
In the case $\dim H<\infty$ the proof of Theorem \ref{th5.8} can be found in
\cite{Nai,DunSch,Sht57}. In the case $\dim H=\infty$ a somewhat weaker result
(only the inverse transform \eqref{5.12}) is contained in \cite{Bru74}. In this
connection note that in the case $\dim H=\infty$ the piecewise holomorphy of a
function $g(\cd)$ is essential, because otherwise the integral in \eqref{5.11}
may not exist.

Our next goal is to obtain a description of all spectral functions $\S_
{\tau,W}(\cd)$ immediately in terms of a boundary parameter $\tau$. Namely,
using the block representations \eqref{3.21} and \eqref{3.22} of the Weyl
functions $M_\pm(\cd)$ introduce the operator functions $\Om_{\tau_0}(\l)(\in
[H^{2n}]), \; S_+(\l)$ $(\in [\cH_0,H^{2n}])$ and $S_-(z)(\in [\cH_1,H^{2n}])$
by setting
\begin{gather}
\Om_{\tau_0}(\l)=\begin{pmatrix} m(\l) & -\tfrac 1 2 I_{H^n}\cr  -\tfrac 1 2
I_{H^n} & 0 \end{pmatrix}: H^n \oplus H^n \to H^n \oplus H^n, \quad
\l\in\CR\label{5.15}\\
S_+(\l)=\begin{pmatrix} -m(\l)& -M_{2+}(\l) \cr I_{H^n} & 0
\end{pmatrix}:H^n\oplus\cH_0'\to H^n\oplus H^n,
\quad \l\in\bC_+\label{5.16}\\
S_-(z)=\begin{pmatrix} -m(z)& -M_{2-}(z) \cr I_{H^n} & 0
\end{pmatrix}:H^n\oplus\cH_1'\to H^n\oplus H^n, \quad z\in\bC_-.\label{5.17}
\end{gather}
Note that $\Om_{\tau_0}(\l)$ is a characteristic matrix  corresponding to the
collection $\tau_0=\{\tau_{0+},\tau_{0-} \}\in\RH$ with
$\tau_{0+}=\{0\}\oplus\cH_1(\in\CA)$.
\begin{theorem}\label{th5.9}
Let the assumptions of Theorem \ref{th5.7} be satisfied and let
$\Om_{\tau_0,W}(\l)(\in [\wt H]), \; S_{W,+}(\l)(\in [\cH_0, \wt H])$ and
$\S_{W,-}(z)(\in [\cH_1, \wt H])$ be the operator functions given by
\begin {gather*}
\Om_{\tau_0,W}(\l)=W^{-1}\Om_{\tau_0}(\l)W^{-1*}, \;\;\l\in\CR; \\
S_{W,+}(\l)=W^{-1}S_+(\l), \;\; \l\in\bC_+; \qquad S_{W,-}(z)=W^{-1}S_-(z),
\;\; z\in\bC_-.
\end{gather*}
Then for each collection $\pair\in\RH$ the equality
\begin {equation}\label{5.18}
\Om_{\tau,W}(\l)=\Om_{\tau_0,W}(\l)-S_{W,+}(\l)(\tau_+(\l)+M_+(\l))^{-1}
S_{W,-}^*(\ov\l), \qquad \l\in \bC_+
\end{equation}
together with \eqref{5.9} defines a (unique) spectral function $\S_ {\tau,W}
(\cd)$ of the boundary problem \eqref{3.27}-\eqref{3.29} corresponding to the
solution $Y_W(\cd,\l)$. Moreover a spectral function $\S_ {\tau,W} (\cd)$ is
orthogonal if and only if $\tau\in\RZ$.
\end{theorem}
\begin{proof}
According to \cite{Mog10} for each collection $\pair\in\RH$ the corresponding
characteristic matrix $\Om_\tau(\cd)$ is given by
\begin {equation}\label{5.18a}
\Om_{\tau}(\l)=\Om_{\tau_0}(\l)-S_{+}(\l)(\tau_+(\l)+M_+(\l))^{-1}
S_{-}^*(\ov\l), \qquad \l\in \bC_+
\end{equation}
This and Theorem \ref{th5.7} yield the desired statement.
\end{proof}
\subsection{Minimal spectral functions} We start the subsection  with the
following lemma which is immediate from Theorem \ref{th5.7}.
\begin{lemma}\label{lem5.10}
Let $\S_{\tau,\f}:\bR\to [\wt\cK]$ be a spectral function of the boundary
problem \eqref{3.27}-\eqref{3.29}, corresponding to the  solution $\f(t,\l)(\in
[\wt\cK, H])$ of the equation \eqref{3.4} (see Definition \ref{def5.5}). Assume
also that $\wt H\supset \wt\cK, \; \wt\cK^\perp=\wt H\ominus\wt\cK$ and
$Y_W(\cd,\l)(\in [\wt H, H])$ is a solution of \eqref{3.4} satisfying the
conditions of Theorem \ref{th5.7} and the equality
$Y_W(t,\l)\up\wt\cK=\f(t,\l)$ (such a solution exists because $0\in\hat\rho
(\wt\f (0,\l))$). Then the (unique)  spectral function  of the boundary problem
\eqref{3.27}-\eqref{3.29} corresponding to $Y_W(\cd,\l)$ is
\begin {equation}\label{5.19}
\S_{\tau,W}(s)=\begin{pmatrix} \S_{\tau,\f}(s) & 0 \cr 0 & 0
\end{pmatrix}:\wt\cK\oplus\wt\cK^\perp \to \wt\cK\oplus\wt\cK^\perp,
\end{equation}
which implies that the spectral function $\S_{\tau,\f}$ is unique.

Conversely if a spectral function $\S_{\tau,W}$ is of the form \eqref{5.19},
then $\S_{\tau,\f}(s)$ is a spectral function corresponding to $\f(\cd,\l)$.
\end{lemma}
Now combining Theorems \ref{th5.7}, \ref{th5.8} with Lemma \ref{lem5.10} and
taking the equality \eqref{4.50} into account one derives the following
theorem.
\begin{theorem}\label{th5.11}
Let $N=(N_0\;\;N_1)$ be an admissible operator pair \eqref{4.1} and let
$\f_N(t,\l)(\in [\hat\cK,H])$ be the operator solution of the equation
\eqref{3.4} with the initial data \eqref{4.34}. Then: 1) for each collection
$\cP=\CD\in\TR$ of holomorphic pairs \eqref{4.2}-\eqref{4.5} there exists a
unique spectral function $\S_{\cP,N}:\bR\to [\hat\cK]$ of the boundary problem
\eqref{4.16}-\eqref{4.20} corresponding to $\f_N(\cd,\l)$. This function is
given by
\begin {equation}\label{5.20}
\S_{\cP, N}(s)=s-\lim_{\d\to+0} w-\lim_{\e\to +0}\frac 1 \pi\int_{-\d}^{s-\d}
Im \;m_\cP(\s+ i\,\e)\,d\s,
\end{equation}
where $m_\cP(\l)$ is the $m$-function corresponding to the boundary problem
\eqref{4.16}-\eqref{4.20}. Moreover, the spectral function $\S_{\cP, N}$ is
orthogonal if and only if $\cP\in\TRZ$.

2) let $\S_{\cP, N}(\cd)$ be a spectral function and let $V:\gH\to L_2(\S_{\cP,
N};\hat\cK)$ be an isometry given by the Fourier transform \eqref{5.6} with
$\f(t,s)=\f_N(t,s)$. Then
\begin {equation*}
(V^*g)(t)=\int_\bR \f_N(t,s)\, d\S_{\cP,N}(s)g(s), \quad g=g(s)\in Hol^0(\hat
\cK).
\end{equation*}
In particular  the inverse Fourier transform is
\begin {equation*}
f(t)=\int_\bR \f_N(t,s)\, d\S_{\cP,N}(s)g_f(s).
\end{equation*}
\end{theorem}
In the next theorem we give a parameterization of all spectral functions
$\S_{\cP, N} (\cd)$ in terms of a boundary parameter $\cP\in \TR$.
\begin{theorem}\label{th5.12}
Let the assumptions of Theorem \ref{th5.11} be satisfied , let $\hat N$ be the
operator \eqref{4.58a} and let $T_{N,0}:\CR\to[\hat\cK], \;T_{N,+}:\bC_+\to
[\cH_0,\hat\cK]$ and  $T_{N,_-}:\bC_-\to [\cH_1,\hat\cK]$ be the operator
functions defined by
\begin {gather*}
T_{N,0}(\l)=\hat N^*\Om_{\tau_0}(\l)\hat N, \;\;\l\in\CR; \\
T_{N,+}(\l)=\hat N^*S_+(\l), \;\; \l\in\bC_+; \;\;\;\; T_{N,-}(z)=\hat
N^*S_-(z) , \;\; z\in\bC_-.
\end{gather*}
Then for each collection $\cP\in\TR$ given by \eqref{4.2}-\eqref{4.5} the
equality
\begin {equation}\label{5.21}
m_\cP(\l)=T_{N,0}(\l)+T_{N,+}(\l)(C_0(\l)-C_1(\l)M_+(\l))^{-1}C_1(\l)T_{N,-}^*(\ov\l),
\quad \l\in\bC_+
\end{equation}
together with \eqref{5.20} defines a (unique) spectral function $\S_{\cP, N}
(\cd)$ of the boundary problem \eqref{4.16}-\eqref{4.20} corresponding to
$\f_N$. Moreover a spectral function $\S_{\cP, N} (\cd)$ is orthogonal if and
only if $\cP\in\TRZ$.
\end{theorem}
\begin{proof}
Let  $\cP\in\TR$ be defined by \eqref{4.2}-\eqref{4.5} and let
$\tau_+(\l)(\in\CA)$ be the corresponding linear relation \eqref{2.14}. Then by
\eqref{5.18a}, \eqref{4.66a} and the equality
\begin {equation*}
-(\tau_+(\l)+M_+(\l))^{-1}=(C_0(\l)-C_1(\l)M_+(\l))^{-1}C_1(\l), \quad
\l\in\bC_+
\end{equation*}
 the $m$-function $m_\cP(\l)$ can be represented via \eqref{5.21}. This
 together with Theorem \ref{th5.11} yield the required statement.
\end{proof}
Next for a given collection $\pair\in\RH$ defined by \eqref{2.14} and
\eqref{3.25}, \eqref{3.26} consider  the corresponding boundary problem
\eqref{3.27}-\eqref{3.29}. Denote by $d_{min}$ the minimal value of $\dim
\wt\cK$ for the set of all spectral functions $\S_{\tau,\f}:\bR\to [\wt\cK]$ of
this boundary problem (recall that according to Definition \ref{def5.5} each
$\S_{\tau,\f}$ corresponds to some operator solution $\f (t,\l)(\in
[\wt\cK,H])$ of the equation \eqref{3.4}).
\begin{definition}\label{def5.13}
A spectral function $\S(\cd)=\S_{\tau,\f}(\cd):\bR\to [\wt\cK]$ will be called
minimal if $\dim \wt\cK=d_{min}$.
\end{definition}
In the following theorem we give  a  description of all minimal spectral
functions of the "triangular" boundary problem \eqref{4.16}-\eqref{4.20}.
\begin{theorem}\label{th5.14}
Let $\Pi=\bta$ be a decomposing $D$-triplet \eqref{3.10} for $L$, let
$\cP=\CD\in\TR$ be a collection of holomorphic pairs \eqref{4.2}-\eqref{4.5}
and let \eqref{4.16}-\eqref{4.20} be the corresponding boundary problem. Then:

1) $d_{min}=\dim \hat\cK$ and the set of all minimal spectral functions
$\S_{min}(\cd)$ is given by
\begin {equation}\label{5.22}
 \S_{min}(s)=X^* \S_{\cP,N}(s)X,
\end{equation}
where $\S_{\cP,N}(s)$ is the (minimal) spectral function defined in Theorem
\ref{th5.11} and $X$ is an automorphism of the space $\hat\cK$. Moreover, the
minimal spectral function $\S_{min}(s)$ given by \eqref{5.22} corresponds to
the operator solution $\f_{min}(t,\l):=\f_N(t,\l)X^{-1*}$ of the equation
\eqref{3.4}.

2) if $\dim H=\infty$, then $d_{min}(=\dim\hat\cK)=\infty$.
\end{theorem}
\begin{proof}
1) Let $\S_{\tau,\f}:\bR\to [\wt\cK]$ be a spectral function of the problem
\eqref{4.16}-\eqref{4.20} corresponding to the solution $\f(t,\l)(\in
[\wt\cK,H])$ with $\wt\f(0,\l)\equiv \wt \f_0(\in [\wt\cK,H^{2n}])$. Since
$0\in\hat\rho (\wt \f_0)$, there are a Hilbert space $\wt\cK^\perp$ and an
operator $\wt\psi_0\in [\wt\cK^\perp, H^{2n}]$ such that the operator $W=(\wt
\f_0\;\; \wt\psi_0)$ is an isomorphism of the  space $\wt H:=\wt\cK\oplus
\wt\cK^\perp$ onto $H^{2n}$.

Let $\Om_{\tau,W}(\l)$ be the operator function \eqref{5.9a} and let
$\S_{\tau,W}(\cd)$ be the spectral function \eqref{5.9} corresponding to the
solution $Y_W(\cd,\l)$ (see Theorem \ref{th5.7}). It follows from \eqref{3.37}
that $s-\lim\limits_{y\to \infty}\Om_{\tau,W}(iy)/y=0$. This and the integral
representation of the Nevanlinna function $\Om_{\tau,W}(\l)$ \cite{Br,KacKre}
yield
\begin {equation}\label{5.23}
\Ker\, Im \,\Om_{\tau,W}(\l)=\{\wt h\in\wt H: \S_{\tau,W}(s)\wt h=0,
\;s\in\bR\}, \quad \l\in\CR.
\end{equation}
Moreover by Lemma \ref{lem5.10} the function $\S_{\tau,W}(s)$ satisfies
\eqref{5.19}, which in view of \eqref{5.23} gives the inclusion
$\wt\cK^\perp\subset \Ker\, Im \,\Om_{\tau,W}(\l), \; \l\in\CR$. Now, letting
$\wt H_0:=\wt H\ominus \Ker\, Im \,\Om_{\tau,W}(\l)$ one obtains $\dim \wt
H_0\leq \dim \wt\cK$.

Next assume that $W'\in [\hat\cK\oplus \hat\cK^\perp, H^{2n}]$ is the
isomorphism \eqref{4.42} and $\Om_{\tau,W'}(\l)$ is the operator function
\eqref{4.48}.  It follows from \eqref{5.9a} that there exists an isomorphism
$C\in [\hat\cK\oplus \hat\cK^\perp, \wt H]$ such that $\Om_{\tau,W'}(\l)=C^*
\Om_{\tau,W}(\l)C$. Moreover, by the block representation \eqref{4.50} one has
$\Ker\, Im \,\Om_{\tau,W'}(\l)=\hat\cK^\perp$.  Hence $\Ker\, Im \,\Om_{\tau,W}
(\l)=C \hat\cK^\perp$ and consequently $\hat\cK=C^*\wt H_0$. Therefore $\dim
\hat\cK=\dim \wt H_0\leq \dim \wt\cK$, which yields the equality $d_{min}=\dim
\hat\cK$.

To prove the relation \eqref{5.22} note that for each automorphism $X\in
[\hat\cK]$ this relation defines the minimal spectral function $\S_{min}(s)=
\S_{\tau,\f_{min}}(s)$, corresponding to the solution $\f_{min}(t,\l):
=\f_N(t,\l)X^{-1*}$. Conversely, let $\S_{min}(s)=\S_{\tau,\f_{min}}(s)$ be a
minimal spectral function corresponding to the solution $\f_{min}(t,\l)(\in
[\hat\cK, H])$. Since $0\in \hat\rho (\wt\f (0,\l))\cap \hat\rho (\wt\f_N
(0,\l))$, there exists an automorphism $X\in [\hat\cK]$ such that $\f_{min}
(t,\l)=\f_N(t,\l)X^{-1*}$ and hence the distribution $\S(s):= X^*
\S_{\cP,N}(s)X$ is a spectral function corresponding to $\f_{min}$. Since by
Lemma \ref{lem5.10} such a function is unique, it follows that
$\S_{min}(s)=\S(s)= X^* \S_{\cP,N}(s)X$.

The statement 2) is implied by the statement 1) and the inequality \eqref{4.10}
\end{proof}
Finally by using the above results we can estimate the spectral multiplicity of
an exit space extension $\wt A\supset L_0$. Namely, the following corollary is
valid.
\begin{corollary}\label{cor5.15}
Let the assumptions of Theorem \ref{th5.14} be satisfied and let
$R_\cP(\l)=P_\gH (\wt A-\l)^{-1}\up\gH$ be a generalized resolvent generated by
the boundary problem \eqref{4.16}-\eqref{4.20}. Then the spectral multiplicity
of the  extension $\wt A$ does nod exceed  $d_{min}(=\dim \hat\cK)$.
\end{corollary}
\begin{proof}
Let $\S=\S_{\cP,N}:\bR\to [\hat \cK]$ be a spectral function defined in Theorem
\ref{th5.11} and let $\chi' (s) $ be a bounded linear map in $Hol (\S,\hat\cK)$
given for all $s\in\bR$ by
\begin {equation*}
(\chi'(s)f)(\s)=\chi_{(-\infty,s)}(\s)f(\s), \quad f=f(\s)\in Hol (\S,\hat\cK)
\end{equation*}
(here $\chi_{(-\infty,s)}(\cd)$ is the indicator of the interval
$(-\infty,s)$). It is easily seen that the map $\chi'(s)$ admits the continuous
extension $\chi (s)\in [L_2(\S;\hat\cK)]\;$ $(s\in\bR) $ such that $\chi (\cd)$
is an orthogonal spectral function (resolution of identity) in
$L_2(\S;\hat\cK)$.

Next assume that $V\in [\gH, L_2(\S;\hat\cK)]$ is an isometry given by the
Fourier transform \eqref{5.6} with $\f=\f_N$ and  let $\cL:=V\gH,
\;\wt\cL=\text{span} \{\cL,\, \chi(s)\cL:s\in\bR\}$. As is known the subspace
$\wt\cL$ reduces the spectral function $\chi (s)$ and the equality $\wt\chi
(s)=\chi (s)\up \wt\cL$ defines the minimal orthogonal  spectral function
$\wt\chi (s)$ in $\wt\cL$ (actually one can prove that
$\wt\cL=L_2(\S;\hat\cK)$). Moreover, the relation \eqref{5.7} yields
\begin {equation}\label{5.25}
F_\cP(t)=V^*\,\chi (t)\,V=V^*(P_\cL \wt\chi (t)\up\cL)V, \quad t\in\bR,
\end{equation}
where $F_\cP(t)=P_\gH \wt E (t)\up\gH$ and $\wt E(t)$ is the orthogonal
spectral function of $\wt A$. It follows from \eqref{5.25} that the spectral
functions $F_\cP(t)$ and $P_\cL \wt\chi (t)\up\cL$ are unitary equivalent and,
consequently, so are the (minimal) orthogonal spectral functions $\wt E (t)$
and $\wt\chi(t)$. This and the fact that $\wt\chi (t)$ is a part of $\chi (t)$
imply that the spectral multiplicity of $\wt E(t)$ does not exceed  the
spectral multiplicity of $\chi (t)$, which in turn does not exceed
$\dim\hat\cK$. This proves the required statement.
\end{proof}
\begin{remark}\label{rem5.16}
It follows from Proposition \ref{pr4.4a} that in the case $n_{b+}<\infty$ (in
particular, $\dim H<\infty$) the statements of Theorem \ref{th5.14} and
Corollary \ref{cor5.15} can be naturally extended to the boundary problems
\eqref{3.27}-\eqref{3.29} generated by a quasi-constant Nevanlinna collection
$\CD$.
\end{remark}

\end{document}